\documentclass[oneside,english]{amsart}
\usepackage[T1]{fontenc}
\usepackage[latin9]{inputenc}
\usepackage{geometry}
\usepackage{color}
\usepackage{float}
\usepackage{amsthm}
\usepackage{amssymb}
\usepackage{graphicx}

\makeatletter
\numberwithin{equation}{section}
\numberwithin{figure}{section}
  \theoremstyle{plain}
  \newtheorem*{thm*}{\protect\theoremname}
\theoremstyle{plain}
\newtheorem{thm}{\protect\theoremname}
  \theoremstyle{plain}
  \newtheorem{cor}[thm]{\protect\corollaryname}
  \theoremstyle{plain}
  \newtheorem{lem}[thm]{\protect\lemmaname}
  \theoremstyle{remark}
  \newtheorem*{rem*}{\protect\remarkname}
 \theoremstyle{definition}
 \newtheorem*{defn*}{\protect\definitionname}

\makeatother

\usepackage{babel}
  \providecommand{\corollaryname}{Corollary}
  \providecommand{\definitionname}{Definition}
  \providecommand{\lemmaname}{Lemma}
  \providecommand{\remarkname}{Remark}
  \providecommand{\theoremname}{Theorem}
\providecommand{\theoremname}{Theorem}

\begin{document}

\title[Isoptic point and Simson line of a quadrilateral]{The perpendicular bisectors construction, the isoptic point and
the Simson line of a quadrilateral }

\author{Olga Radko and Emmanuel Tsukerman}
\begin{abstract}
Given a noncyclic quadrilateral, we consider an iterative procedure
producing a new quadrilateral at each step. At each iteration, the
vertices of the new quadrilateral are the circumcenters of the triad
circles of the previous generation quadrilateral.\textcolor{red}{{}
}\textcolor{black}{The main goal of the paper is to prove a number
of interesting properties of the limit point of this iterative process.
We show that the limit point is the common center of spiral similarities
taking any of the triad circles into another triad circle. As a consequence,
the point has the isoptic property (i.e., all triad circles are visible
from the limit point at the same angle). Furthermore, the limit point
can be viewed as a generalization of a circumcenter. It also has properties
similar to those of the isodynamic point of a triangle.} We also characterize
the limit point as the unique point for which the pedal quadrilateral
is a parallelogram. Continuing to study the pedal properties with
respect to a quadrilateral, we show that for every quadrilateral there
is a unique point (which we call the Simson point) such that its pedal
consists of four points on line, which we call the Simson line, in
analogy to  the case of a triangle. Finally, we define a version of
isogonal conjugation for a quadrilateral and prove that the isogonal
conjugate of the limit point is a parallelogram, while that of the
Simson point is a degenerate quadrilateral whose vertices coincide
at infinity. 
\end{abstract}
\maketitle

\section{Introduction\label{sec:Introduction}}

\textcolor{black}{The perpendicular bisector construction that we
investigate in this paper arises very naturally in an attempt to find
a replacement for a circumcenter in the case of a noncyclic quadrilateral
$Q^{(1)}=A_{1}B_{1}C_{1}D_{1}$. Indeed, while there is no circle
going through all four vertices, for every triple of vertices there
is a unique circle (called the }\textcolor{black}{\emph{triad circle}}\textcolor{black}{)
passing through them. The centers of these four triad circles can
be taken as the vertices of a new quadrilateral, and the process can
be iterated to obtain a sequence of noncyclic quadrilaterals: $Q^{(1)},Q^{(2)},Q^{(3)},\dots.$ }

To reverse the iterative process, one finds isogonal conjugates of
each of the vertices with respect to the triangle formed by the remaining
vertices of the quadrilateral.

It turns out that all odd generation quadrilaterals are similar, and
all even generation quadrilaterals are similar. Moreover, there is
a point that serves as the center of spiral similarity for any pair
of odd generation quadrilaterals as well as for any pair of even generation
quadrilaterals. The angle of rotation is $0$ or $\pi$ depending
on whether the quadrilateral is concave or convex, and the ratio $r$
of similarity is\textcolor{blue}{{} }\textcolor{black}{a constant that
is negative for convex noncyclic quadrilaterals, zero for cyclic quadrilaterals,
and $\geq1$ for concave quadrilaterals.} If $|r|\neq1$, the same
special point turns out to be the limit point for the iterative process
or for the reverse process. 

The main goal of this paper is to prove the following
\begin{thm*}
For each quadrilateral $Q^{(1)}=A_{1}B_{1}C_{1}D_{1}$ there is a
unique point $W$ that has any (and, therefore, all) of the following
properties:
\begin{enumerate}
\item \label{enu:ss center}$W$ is the center of the spiral similarity
for any two odd (even) generation quadrilaterals in the iterative
process;
\item \label{enu: ratio}Depending on the value of the ratio of similarity
in the iterative process, there are the following possibilities:

\begin{enumerate}
\item If $|r|<1$, the quadrilaterals in the iterated perpendicular bisectors
construction converge to $W$;
\item If $|r|=1$, the iterative process is periodic (with period $2$ or
$4$); $W$ is the common center of rotations for any two odd (even)
generation quadrilaterals;
\item If $|r|>1$, the quadrilaterals in the reverse iterative process (obtained
by isogonal conjugation) converge to $W$;
\end{enumerate}
\item \label{enu: W on CSs}$W$ is the common point of the six circles
of similitude $CS(o_{i},o_{j})$ for any pair of triad circles $o_{i},o_{j}$,
$i,j\in\{1,2,3,4\}$, where $o_{1}=(D_{1}A_{1}B_{1})$, $o_{2}=(A_{1}B_{1}C_{1})$,
$o_{3}=(B_{1}C_{1}D_{1})$, $o_{4}=(C_{1}D_{1}A_{1})$. 
\item \label{enu:isoptic}(isoptic property) Each of the triad circles is
visible from $W$ at the same angle.
\item \label{enu:circumcenter}(generalization of circumcenter) The (directed)
angle subtended by any of the quadrilateral's sides at $W$ equals
to the sum of the angles subtended by the same side at the two remaining
vertices.
\item \label{enu:isodynamic}(isodynamic property) The distance from $W$
to any vertex is inversely proportional to the radius of the triad
circle determined by the remaining three vertices. 
\item \label{enu: W as inversion}$W$ is obtained by inversion of any of
the vertices of the original quadrilateral in the corresponding triad-circle
of the second generation:
\[
W=\mbox{Inv}_{o_{1}^{(2)}}(A)=\mbox{Inv}_{o_{2}^{(2)}}(B)=\mbox{Inv}_{o_{3}^{(2)}}(C)=\mbox{Inv}_{o_{4}^{(2)}}(D),
\]
where $ $$o_{1}^{(2)}=(D_{2}A_{2}B_{2})$, $o_{2}^{(2)}=(A_{2}B_{2}C_{2})$,
$o_{3}^{(2)}=(B_{2}C_{2}D_{2})$, $o_{4}^{(2)}=(C_{2}D_{2}A_{2})$. 
\item \label{enu: W isogonal conj}$W$ is obtained\textcolor{blue}{{} }\textcolor{black}{by
composition of isogonal conjugation of a vertex in the triangle formed
by the remaining vertices and inversion in the circumcircle of that
triangle.}
\item \label{enu:W on CSs of all generations} $W$ is the center of spiral
similarity for any pair of triad circles (of possibly different generations).
That is, $W\in CS(o_{i}^{(k)},\ o_{j}^{(l)})$ for all $i,j,k,l$.
\item \label{enu:pedal of W}The pedal quadrilateral of $W$ is a (nondegenerate)
parallelogram. Moreover, its angles equal to the angles of the Varignon
parallelogram. 
\end{enumerate}
\end{thm*}
Many of these properties of $W$ were known earlier. In particular,
several authors (G.T. Bennett in an unpublished work, De Majo \cite{De Majo},
H.V. Mallison \cite{Mallison}) have considered a point that is defined
as the common center of spiral similarities. Once the existence of
such a point is established, it is easy to conclude that all the triad
circles are viewed from this point under the same angle (this is the
so-called \emph{isoptic property}). Since it seems that the oldest
reference to the point with such an isoptic property is to an unpublished
work of G.T. Bennett given by H.F. Baker \cite{Baker} in his {}``Principles
of Geometry'' (vol. 4), in 1925, we propose to call the center of
spiral similarities in the iterative process \emph{Bennett's isoptic
point.} 

C.F. Parry and M.S. Longuet-Higgins \cite{Reflections 3} showed the
existence of a point with property \ref{enu: W as inversion} using
elementary geometry. 

\textcolor{black}{Mallison \cite{Mallison} defined $W$ using property
\ref{enu: W on CSs} and credited T. McHugh for observing that this
implies property \ref{enu:circumcenter}. }

\textcolor{black}{Several authors, including Wood \cite{Wood} and
De Majo \cite{De Majo}, have looked at the properties of the isoptic
point from the point of view of the unique rectangular hyperbola going
through the vertices of the quadrilateral, and studied its properties
related to cubics. For example, P.W. Wood \cite{Wood} considered
the diameters of the rectangular hyperbola that go through $A,B,C,D$.
Denoting by $\bar{A},\bar{B,}\bar{C},\bar{D}$ the other endpoints
of the diameters, he showed that the isogonal conjugates of these
points in triangles $\triangle BCD$, $\triangle CAD$,$\triangle ABD$,
$\triangle ABC$ coincide. Starting from this, he proved properties
\ref{enu:isoptic} and \ref{enu: W as inversion} of the theorem.
He}\textcolor{blue}{{} }\textcolor{black}{also mentions the reversal
of the iterative process using isogonal conjugation (Also found in
\cite{Wood}, \cite{Scimemi}, \cite{Grinberg}). Another interesting
property mentioned by Wood is that $W$ is the Fregier point of the
center of the rectangular hyperbola for the conic $ABCDO$, where
O is the center of the rectangular hyperbola. }

De Majo \cite{De Majo} uses the property that inversion in a point
on the circle of similitude of two circles transforms the original
circles into a pair of circles whose radii are inversely proportional
to those of the original circles to show that that there is a common
point of intersection of all $6$ circles of similitude. He describes
the iterative process and states property \ref{enu:ss center}, as
well as several other properties of $W$ (including \ref{enu: W isogonal conj}).
Most statements are given without proofs. 

\textcolor{black}{Scimemi \cite{Scimemi} describes a M\"{o}bius
transformation that characterizes $W$: there exists a line going
through $W$ and a circle centered at $W$ such that the product of
the reflection in the line with the inversion in the circle maps each
vertex of the first generation into a vertex of the second generation. }

\textcolor{black}{The fact that the third generation quadrilateral
is similar to the original quadrilateral and the expression for the
ratio of similarity has appeared in the form of a problem by V.V.
Prasolov in \cite{Prasolov,Prasolov-translation}. Independently,
the same question was formulated by J. Langr \cite{Langr} , and the
expression for the ratio (under certain conditions) was obtained by
D. Bennett \cite{D. Bennett} (apparently, no relation to G.T. Bennett
mentioned above), and J. King \cite{King}. A paper by G.C. Shepard
\cite{Shepard} found an expression for the ratio as well. (See \cite{cut-the-knot}
for a discussion of these works). }

Properties \ref{enu:W on CSs of all generations} and \ref{enu:pedal of W}
appear to be new. 

\textcolor{black}{For the convenience of the reader, we give a complete
and self-contained exposition of all the properties in the Theorem
above, as well as proofs of several related statements. }

In addition to investigating properties of $W$, we show that there
is a unique point for which the feet of the perpendiculars to the
sides lie on a straight line. In analogy with the case of a triangle,
we call this line the \emph{Simson} \emph{line} of a quadrilateral
and the point --- the \emph{Simson point.}\textcolor{black}{{} The existence
of such a point is stated in \cite{Johnson-book} where it is obtained
as the intersection of the Miquel circles of the complete quadrilateral. }

Finally, we introduce a version of isogonal conjugation for a quadrilateral
and show that the isogonal conjugate of $W$ is a parallelogram\textcolor{green}{,
}\textcolor{black}{and that of the Simson point is a degenerate quadrilateral
whose vertices are at infinity, in analogy with the case of the points
on the circumcircle of a triangle.}

\section{The iterative process\label{sec: Iterative Process}}

Let $A_{1}B_{1}C_{1}D_{1}$ be a quadrilateral. If $A_{1}B_{1}C_{1}D_{1}$
is cyclic, the center of the circumcircle can be found as the intersection
of the $4$ perpendicular bisectors to the sides of the quadrilateral. 

Assume that $Q^{(1)}=A_{1}B_{1}C_{1}D_{1}$ is a noncyclic quadrilateral%
\footnote{Sometimes we drop the lower index $1$ when denoting vertices of $Q^{(1)}$,
so $ABCD$that and $A_{1}B_{1}C_{1}D_{1}$ are used interchangeably
throughout the paper. %
}. Is there a point that, in some sense, plays the role of the circumcenter?
Let $Q^{(2)}=A_{2}B_{2}C_{2}D_{2}$ be the quadrilateral formed by
the intersections of the perpendicular bisectors to the sides of $A_{1}B_{1}C_{1}D_{1}$.
The vertices $A_{2},B_{2},C_{2},D_{2}$ of the new quadrilateral are
the circumcenters of the triad-triangles $\triangle D_{1}A_{1}B_{1}$
, $\triangle A_{1}B_{1}C_{1}$, $\triangle B_{1}C_{1}D_{1}$ and $\triangle C_{1}D_{1}A_{1}$
formed by vertices of the original quadrilateral taken three at a
time.
\begin{figure}[h]
\includegraphics[scale=0.3]{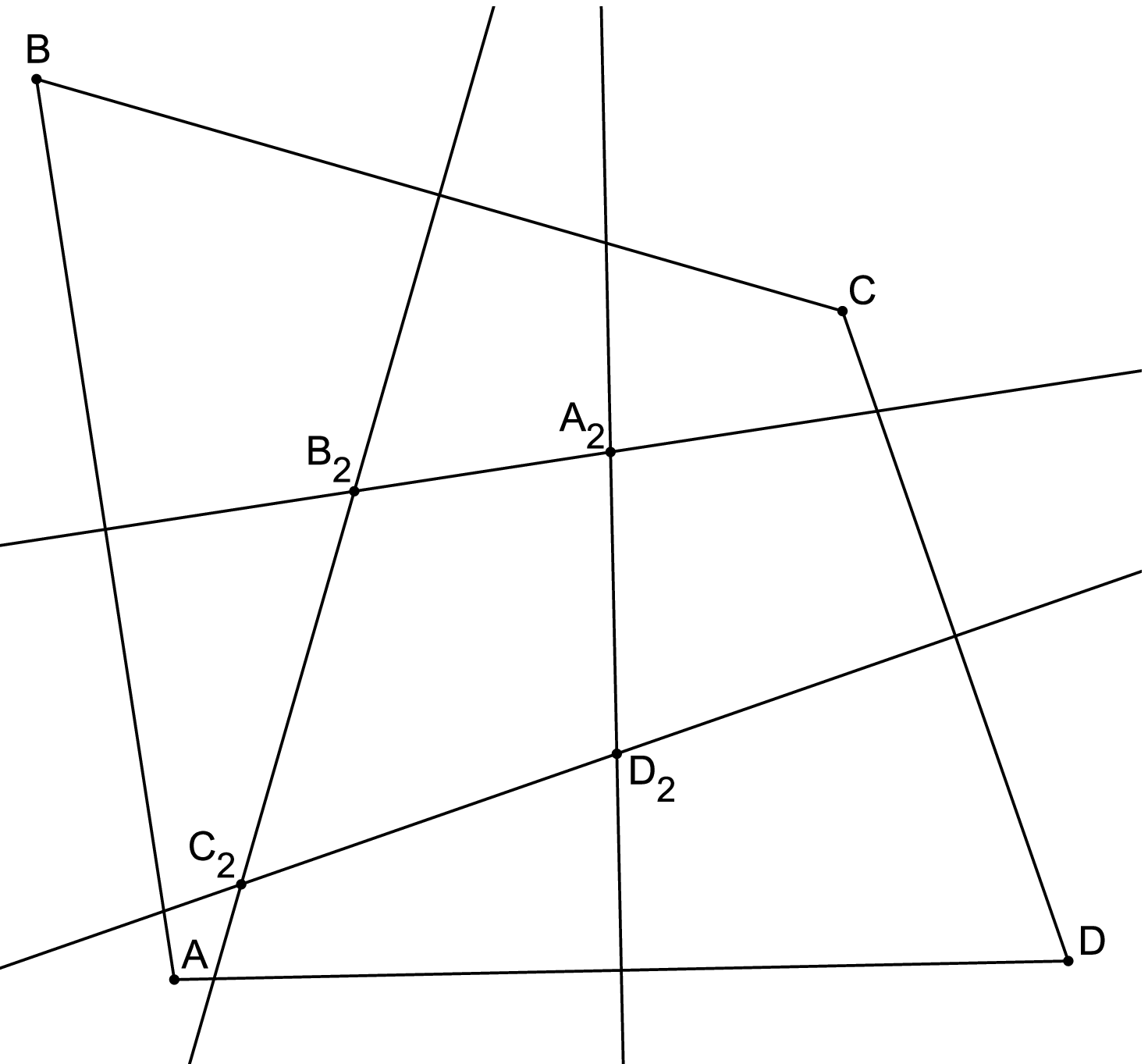}\includegraphics[scale=0.3]{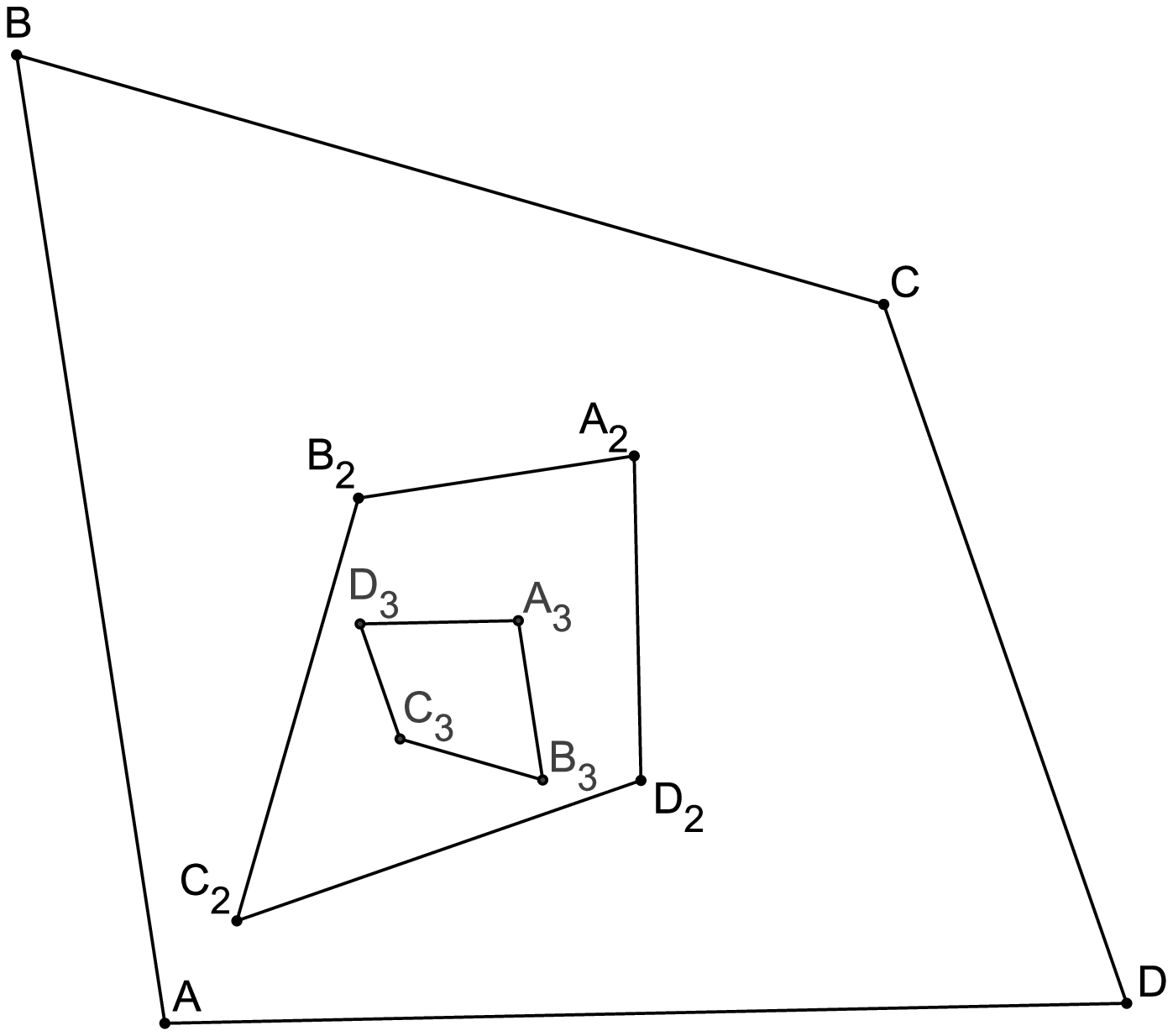}\caption{\label{fig:Iterative Process}The perpendicular bisector construction
and $Q^{(1)},Q^{(2)},Q^{(3)}$.}
\end{figure}

Iterating this process (i.e., constructing the vertices of the next
generation quadrilateral by intersecting the perpendicular bisectors
to the sides of the current one), we obtain the \textcolor{black}{successive}
generations, $Q^{(3)}=A_{3}B_{3}C_{3}D_{3}$, $Q^{(4)}=A_{4}B_{4}C_{4}D_{4}$
and so on, see Figure \ref{fig:Iterative Process}.

The first thing we note about the iterative process is that it can
be reversed using isogonal conjugation. \textcolor{black}{Recall that
given a triangle $\triangle ABC$ and a point $P$, the }\textcolor{black}{\emph{isogonal
conjugate}}\textcolor{black}{{} of $P$ with respect to $\triangle ABC$
(denoted by $\mbox{Iso}_{\triangle ABC}(P)$) is the point of intersection
of the reflections of the lines $AP$, $BP$ and $CP$ in the angle
bisectors of $\angle A$, $\angle B$ and $\angle C$ respectively.
One of the basic properties of isogonal conjugation is that the isogonal
conjugate of $P$ is the circumcenter of the triangle obtained by
reflecting $P$ in the sides of $\triangle ABC$ (see, e.g., \cite{Grinberg}
for more details). This property immediately implies}
\begin{thm}
\label{thm:Reversing using Isogonal Conjugation}The original quadrilateral
$A_{1}B_{1}C_{1}D_{1}$ can be reconstructed from the second generation
quadrilateral $A_{2}B_{2}C_{2}D_{2}$ using isogonal conjugation:
\[
\begin{array}{ccc}
A_{1} & = & \mbox{Iso}_{\triangle D_{2}A_{2}B_{2}}(C_{2}),\\
B_{1} & = & \mbox{Iso}_{\triangle A_{2}B_{2}C_{2}}(D_{2}),\\
C_{1} & = & \mbox{Iso}_{\triangle B_{2}C_{2}D_{2}}(A_{2}),\\
D_{1} & = & \mbox{Iso}_{\triangle C_{2}D_{2}A_{2}}(B_{2}).
\end{array}
\]

\end{thm}
The following theorem describes the basic properties of the iterative
process. 
\begin{thm}
\label{thm:Inverse process}Let $Q^{(1)}$ be a quadrilateral. Then
\begin{enumerate}
\item $Q^{(2)}$ degenerates to a point if and only if $Q^{(1)}$ is cyclic.
\item \label{enu:angles supplementary} If $Q^{(1)}$ is not cyclic, the
corresponding angles of the first and second generation quadrilaterals
are supplementary:
\[
\angle A_{1}+\angle A_{2}=\angle B_{1}+\angle B_{2}=\angle C_{1}+\angle C_{2}=\angle D_{1}+\angle D_{2}=\pi.
\]

\item \label{enu:similarity of generations}If $Q^{(1)}$ is not cyclic,
all odd generation quadrilaterals are similar to each other and all
the even generation quadrilaterals are similar to each other:
\begin{eqnarray*}
 & Q^{(1)}\sim Q^{(3)}\sim Q^{(5)}\sim\dots,\\
 & Q^{(2)}\sim Q^{(4)}\sim Q^{(6)}\sim\dots.
\end{eqnarray*}

\item All odd generation quadrilaterals are related to each other via spiral
similarities with respect to a common center.
\item All even generation quadrilaterals are also related to each other
via spiral similarities with respect to a common center.
\item The angle of rotation for each spiral similarity is $\pi$ (for a
convex quadrilateral) or a $0$ (for a concave quadrilateral). The
ratio of similarity is 
\begin{equation}
r=\frac{1}{4}(\cot\alpha+\cot\gamma)\cdot(\cot\beta+\cot\delta),\label{eq:ratio of similarity}
\end{equation}
where $\alpha=\angle A_{1}$, $\beta=\angle B_{1}$, $\gamma=\angle C_{1}$
and $\delta=\angle D_{1}$ are the angles of $Q^{(1)}$. 
\item The center of spiral similarities is the same for both the odd and
the even generations. 
\end{enumerate}
\end{thm}
\begin{proof}
The first and second statements follow immediately from the definition
of the iterative process. To show that all odd generation quadrilaterals
are similar to each other and all even generation quadrilaterals similar
to each other, it is enough to notice that both the corresponding
sides and the corresponding diagonals of all odd (even) generation
quadrilaterals are pairwise parallel. 

Let $W_{1}\doteq A_{1}A_{3}\cap B_{1}B_{3}$ be the center of spiral
similarity taking $Q^{(1)}$ into $Q^{(3)}$. Similarly, let $W_{2}$
be the center of spiral similarity taking $Q^{(2)}$ into $Q^{(4)}$.
Denote the midpoints of segments $A_{1}B_{1}$ and $A_{3}B_{3}$ by
$M_{1}$ and $M_{3}$. (See fig. \ref{fig:W1 is W2}). To show that
$W_{1}$ and $W_{2}$ coincide, notice that $\triangle B_{1}M_{1}A_{2}\sim\triangle B_{3}M_{3}A_{4}$.
Since the corresponding sides of these triangles are parallel, they
are related by a spiral similarity. Since $B_{1}B_{3}\cap M_{1}M_{3}=W_{1}$
and $M_{1}M_{3}\cap B_{2}B_{4}=W_{2}$, it follows that $W_{1}=W_{2}$.
Let now $W_{3}$ be the center of spiral similarity that takes $Q^{(3)}$
into $Q^{(5)}$. By the same reasoning, $W_{2}=W_{3}$, which implies
that $W_{1}=W_{3}$. Continuing by induction, we conclude that the
center of spiral similarity for any pair of odd generation quadrilaterals
coincides with that for any pair of even generation quadrilaterals.
We denote this point by $W$.
\end{proof}
\begin{figure}[H]
\includegraphics[scale=0.15]{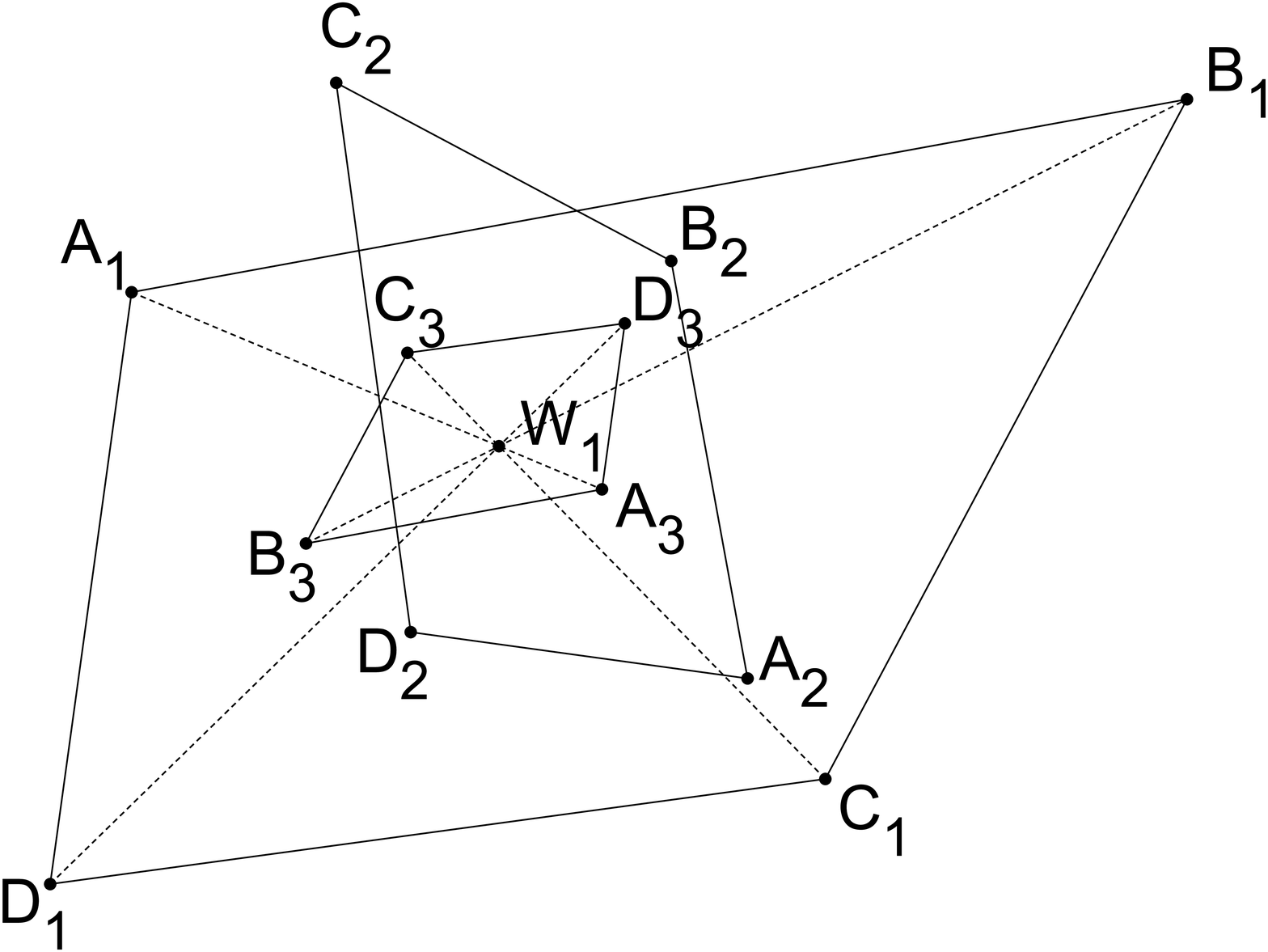}\includegraphics[scale=0.15]{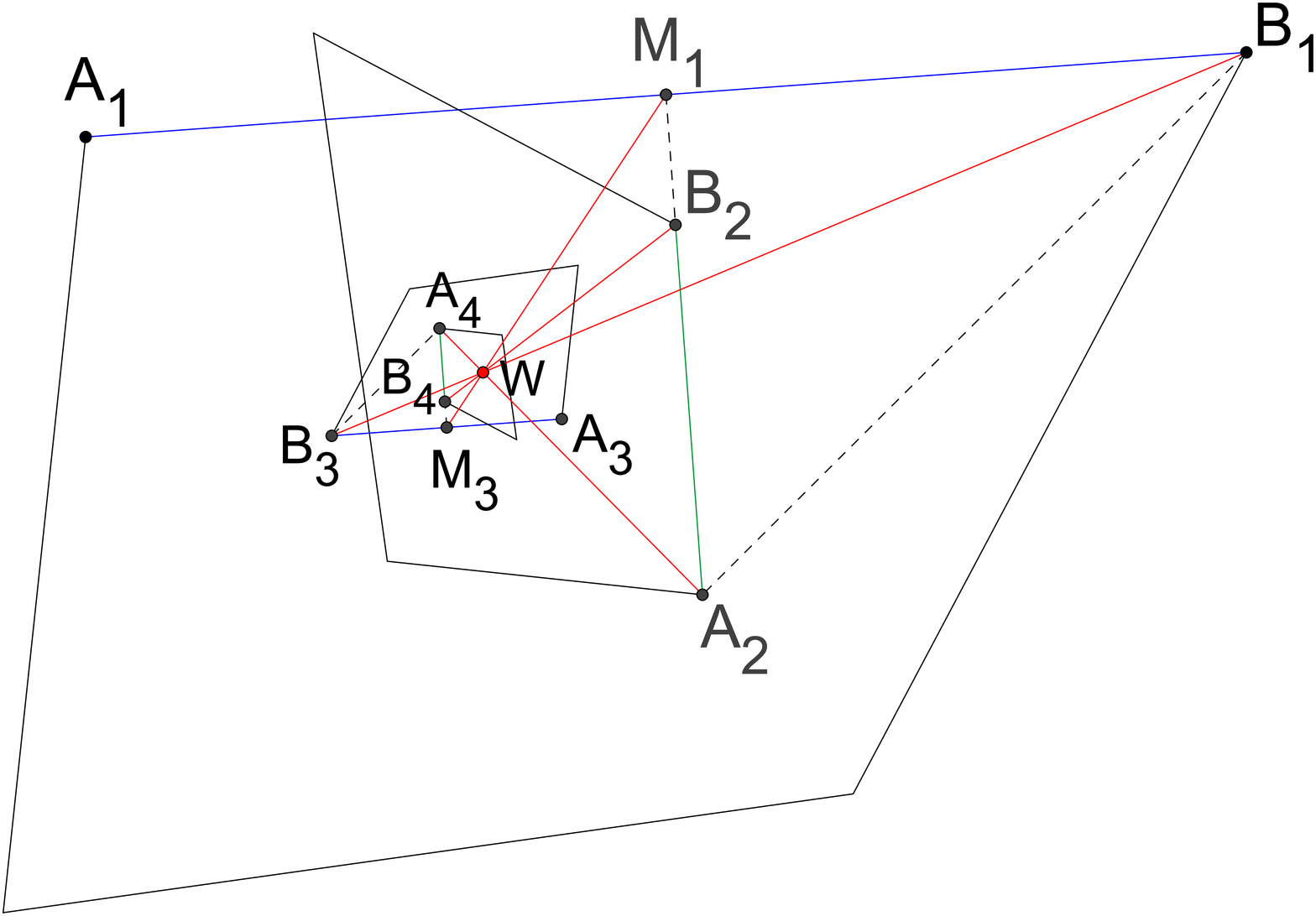}

\caption{\label{fig:W1 is W2}$W$ as the center of spiral similarities.}
\end{figure}

${}$Properties (\ref{enu:angles supplementary}) and (\ref{enu:similarity of generations})
of the theorem above imply the following
\begin{cor}
\label{cor:trapezoid case}The even and odd generation quadrilaterals
are similar to each other if and only if $Q^{(1)}$ is a trapezoid.
\end{cor}
${}$

\textcolor{black}{The ratio of similarity $r=r(\alpha,\beta,\gamma,\delta)$
takes values in $(-\infty,0]\cup[1,\infty)$ and characterizes the
shape of $Q^{(1)}$ in the following way:}
\begin{enumerate}
\item $r\leq0$ if and only if $Q^{(1)}$ is convex. Moreover, $r=0$ if
and only if $Q^{(1)}$ is cyclic.
\item $r\geq1$ if and only if $Q^{(1)}$ is concave. Moreover, $r=1$ if
and only if $Q^{(1)}$ is \emph{orthocentric} (that is, each of the
vertices is the orthocenter of the triangle formed by the remaining
three vertices. Alternatively, an orthocentric quadrilateral is characterized
by being a concave quadrilateral for which the two opposite acute
angles are equal). 
\end{enumerate}
For convex quadrilaterals, $r$ can be viewed as a measure of how
noncyclic the original quadrilateral is. Recall that since the opposite
angles of a cyclic quadrilateral add up to $\pi$, the difference
\begin{equation}
|(\alpha+\gamma)-\pi|=|(\beta+\delta)-\pi|\label{eq:sum of angles non-cyclicity}
\end{equation}
can be taken as the simplest measure of noncyclicity. This measure,
however, treats two quadrilaterals with equal sums of opposite angles
as equally noncyclic. The ratio $r$ provides a refined measure of
noncyclicity. For example, for a fixed sum of opposite angles, $\alpha+\gamma=C$,
$\beta+\delta=2\pi-C$, where $C\in(0,2\pi)$, the convex quadrilateral
with the smallest $|r|$ is the parallelogram with $\alpha=\gamma=\frac{C}{2}$,
$\beta=\delta$. 

Similarly, for concave quadrilaterals, $r$ measures how different
the quadrilateral is from being orthocentric. 

${}$

Since the angles between diagonals are the same for all generations,
it follows that the ratio is the same for all pairs of consecutive
generations: 
\[
\frac{\mbox{Area}(Q^{(n)})}{\mbox{Area}(Q^{(n-1)})}=|r|.
\]
Assuming the quadrilateral is noncyclic, there are the following three
possibilities:
\begin{enumerate}
\item When $|r|<1$ (which can only happen for convex quadrilaterals), the
quadrilaterals in the iterative process converge to $W$.
\item When $|r|>1$, the quadrilaterals in the inverse iterative process
converge to $W$.
\item When $|r|=1$, all the quadrilaterals have the same area. The iterative
process is periodic with period $4$ for all quadrilaterals with $|r|=1$,
except for the following two special cases. If $Q^{(1)}$ is either
a parallelogram with angle $\frac{\pi}{4}$ (so that $r=-1$) or forms
an orthocentric system (so that $r=1$), we have $Q^{(3)}=Q^{(1)}$,
$Q^{(4)}=Q^{(2)}$, and the iterative process is periodic with period
$2$. 
\end{enumerate}
${}$

By setting $r=0$ in formula (\ref{eq:ratio of similarity}), we obtain
the familiar relations between the sides and diagonals of a cyclic
quadrilateral $ABCD$:
\begin{eqnarray}
|AC|\cdot|BD| & = & |AB|\cdot|CD|+|BC|\cdot|AD|,\qquad\mbox{(Ptolemy's theorem)}\label{eq:Ptolemy-1}\\
\frac{|AC|}{|BD|} & = & \frac{|AB|\cdot|AD|+|CB|\cdot|CD|}{|BA|\cdot|BC|+|DA|\cdot|DC|.}\label{eq:Ptolemy-2}
\end{eqnarray}

Since the vertices of the next generation depend only on the vertices
of the previous one (but not on the way the vertices are connected),
one can see that $W$ and $r$ for the (self-intersecting) quadrilaterals
$ACBD$ and $ACDB$ coincide with those for $ABCD$. This observation
allows us to prove the following
\begin{cor}
\label{cor:cotangent identity}The angles between the sides and the
diagonals of a quadrilateral satisfy the following identities:
\begin{eqnarray*}
(\cot\alpha+\cot\gamma)\cdot(\cot\beta+\cot\delta) & = & (\cot\alpha_{1}-\cot\beta_{2})\cdot(\cot\delta_{2}-\cot\gamma_{1}),\\
(\cot\alpha+\cot\gamma)\cdot(\cot\beta+\cot\delta) & = & (\cot\delta_{1}-\cot\alpha_{2})\cdot(\cot\beta_{1}-\cot\gamma_{2})
\end{eqnarray*}
where $\alpha_{i},\beta_{i},\gamma_{i},\delta_{i}$, $i=1,2$ are
the directed angles formed between sides and diagonals of a quadrilateral
(see Figure \ref{fig:angles between sides diagonals}).
\end{cor}
\begin{figure}
\begin{centering}
\includegraphics[scale=0.14]{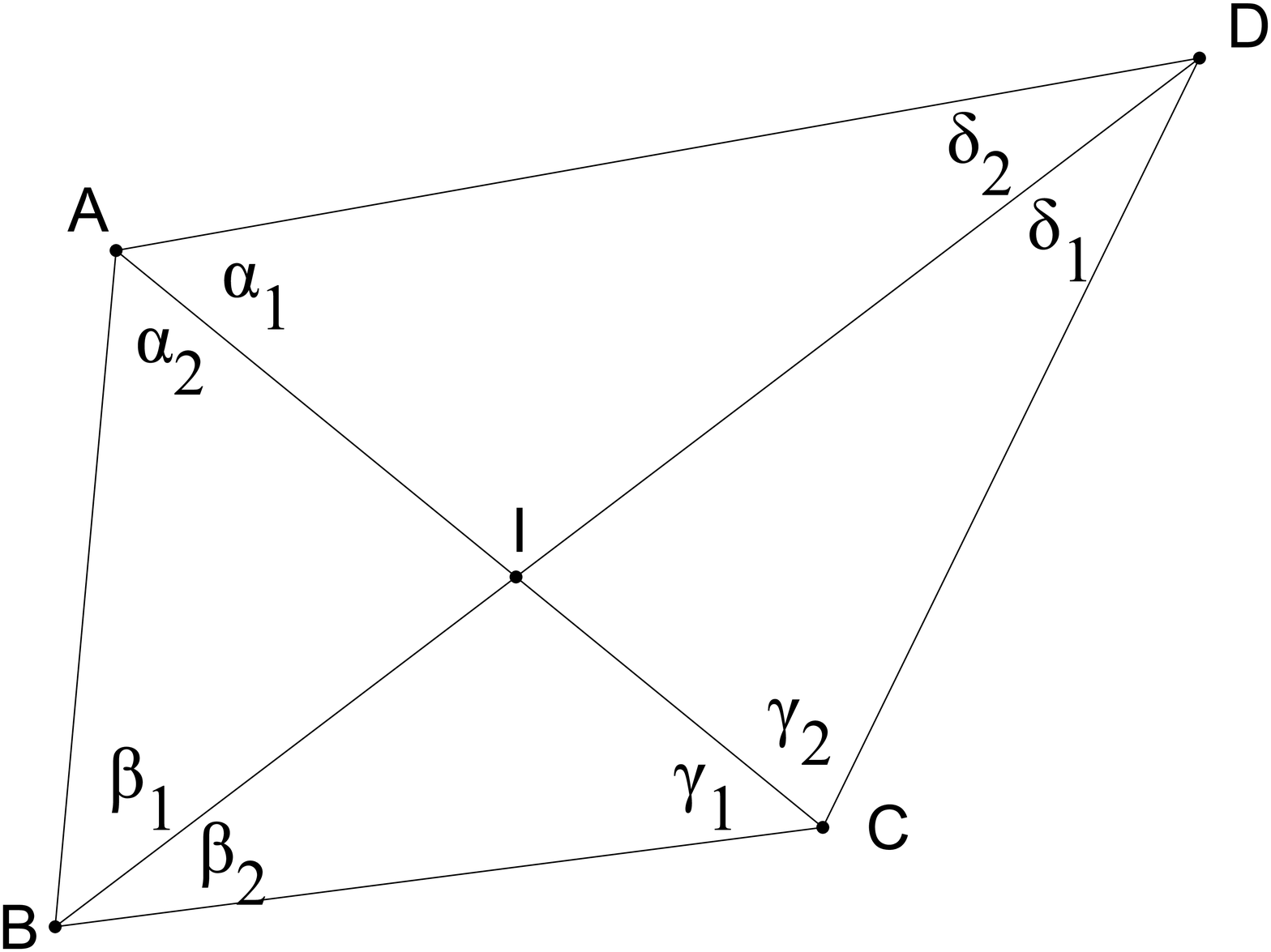}
\par\end{centering}

\caption{\label{fig:angles between sides diagonals}The angles between the
sides and diagonals of a quadrilateral.}
\end{figure}

\begin{proof}
Since the (directed) angles of $ACBD$ are $-\alpha_{1},\beta_{2},\gamma_{1},-\delta_{2}$
and the directed angles of $ACDB$ are $\alpha_{2},\beta_{1},-\gamma_{2},-\delta_{1}$,
the identities follow from formula (\ref{eq:ratio of similarity})
for the ratio of similarity. 
\end{proof}

\section{Properties of the Center of Spiral Similarity\label{sec: W}}

We will show that $W$, defined as the limit point of the iterated
perpendicular bisectors construction in the case that $|r|<1$ (or
of its reverse in the case that $|r|>1$), is the common center of
all spiral similarities taking any of the triad circles into another
triad circle in the iterative process. 

First, we will prove that any of the triad circles of the first generation
quadrilateral can be taken into another triad circle of the first
generation by a spiral similarity centered at $W$ (Theorem \ref{thm:W on CSes}).
This result allows us to view $W$ as a generalization of the circumcenter
for a noncyclic quadrilateral (Corollary \ref{cor:Angle addition for W}
and Corollary \ref{cor:feetCircles}), to prove its isoptic (Theorem
\ref{thm: Isoptic property}), isodynamic (Corollary \ref{cor:Isodynamic property})
and inversive (Theorem \ref{thm:Inversive Property of W}) properties,
as well as to establish some other results. We then prove several
statements that allow us to conclude (see Theorem \ref{thm: W on all CSs})
that $W$ serves as the center of spiral similarities for any pair
of triad circles of any two generations. 

Several objects associated to a configuration of two circles on the
plane will play a major role in establishing properties of $W$. We
will start by recalling the definitions and basic constructions related
to these objects.

\subsection{Preliminaries: circle of similitude, mid-circles and the radical
axis of two circles.\label{sub:CS, MS, RA.}}

Let $o_{1}$ and $o_{2}$ be two (intersecting%
\footnote{most of the constructions remain valid for non-intersecting circles.
However, they sometimes have to be formulated in different terms.
Since we will only deal with intersecting circles, we will restrict
our attention to this case. %
}) circles on the plane with centers $O_{1}$ and $O_{2}$ and radii
$R_{1}$ and $R_{2}$ respectively. Let $A$ and $B$ be the points
of intersection of the two circles. There are several geometric objects
associated to this configuration (see Figure  \ref{fig:CS MC RA}):
\begin{enumerate}
\item The \emph{circle of similitude} $CS(o_{1},o_{2})$ is the set of points
$P$ on the plane such that the ratio of their distances to the centers
of the circles is equal to the ratio of the radii of the circles:
\[
\frac{|PO_{1}|}{|PO_{2}|}=\frac{R_{1}}{R_{2}}.
\]
In other words, $CS(o_{1},o_{2})$ is the Apollonian circle determined
by points $O_{1}$, $O_{2}$ and ratio $R_{1}/R_{2}$.
\item The \emph{radical axis }$RA(o_{1},o_{2})$ can be defined as the line
through the points of intersection. 
\item The two \emph{mid-circles }(sometimes also called the \emph{circles
of antisimilitude}) $MC_{1}(o_{1},o_{2})$ and $MC_{2}(o_{1},o_{2})$
are the circles that invert $o_{1}$ into $o_{2}$, and vice versa:
\[
\mbox{Inv}_{MC_{i}(o_{1},o_{2})}(o_{1})=o_{2},\qquad i=1,2.
\]

\end{enumerate}
\begin{figure}
\includegraphics[scale=0.3]{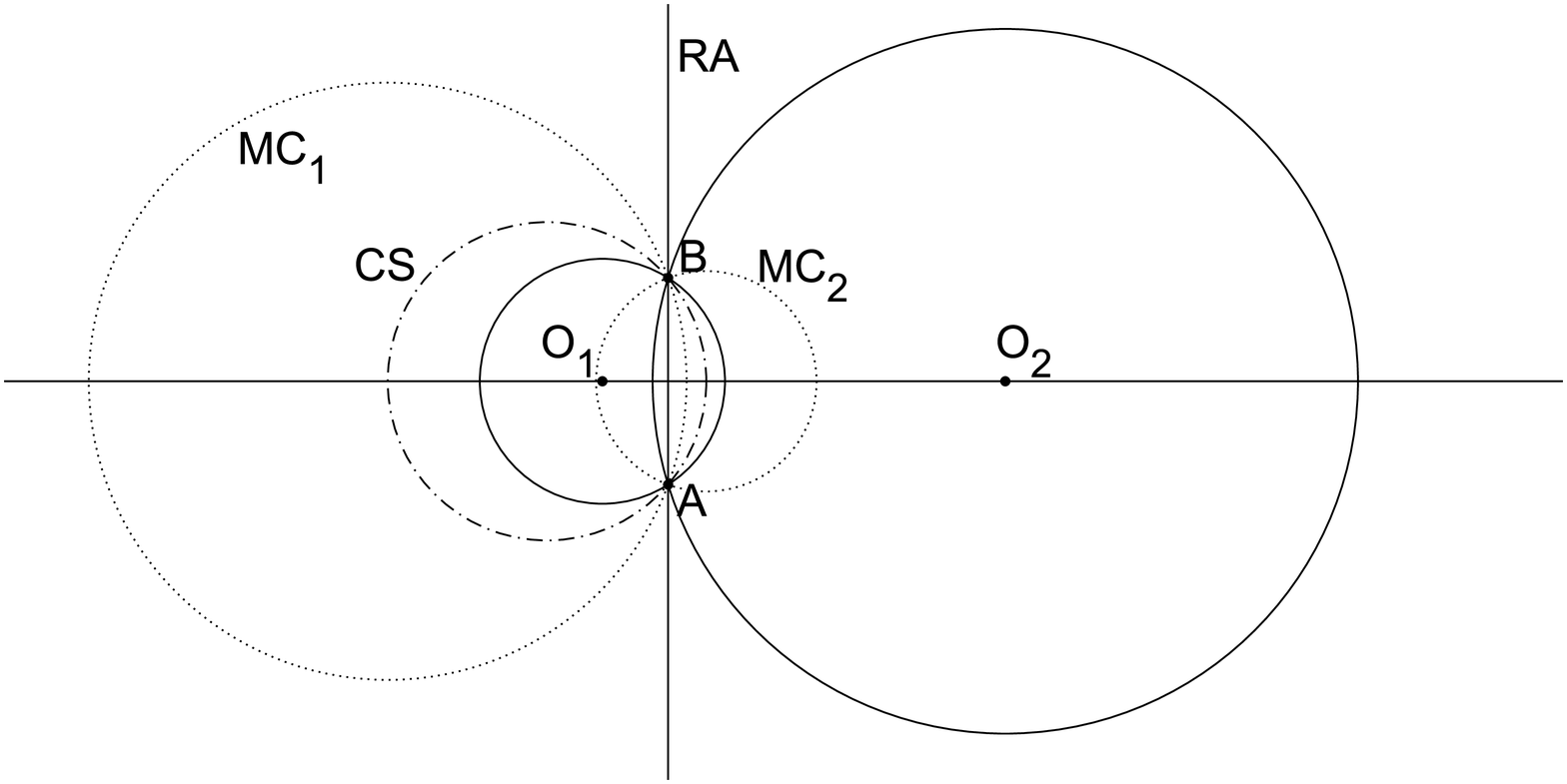}\caption{\label{fig:CS MC RA}Circle of similitude, mid-circles and radical
axis. }
\end{figure}

Here are several important properties of these objects (see \cite{Johnson-book}
and \cite{Coxeter} for more details):
\begin{enumerate}
\item $CS(o_{1},o_{2})$ is the locus of centers of spiral similarities
taking $o_{1}$ into $o_{2}$. For any $E\in CS(o_{1},o_{2})$, there
is a spiral similarity centered at $E$ that takes $o_{1}$ into $o_{2}$.
The ratio of similarity is $R_{2}/R_{1}$ and the angle of rotation
is $\angle O_{1}EO_{2}$. 
\item \label{enu:CS Inversion exchanges centers}Inversion with respect
to $CS(o_{1},o_{2})$ takes centers of $o_{1}$ and $o_{2}$ into
each other:
\[
\mbox{Inv}_{CS(o_{1},o_{2})}(O_{1})=O_{2}.
\]

\item Inversion with respect to any of the mid-circles exchanges the circle
of similitude and the radical axis:
\[
\mbox{Inv}_{MC_{i}(o_{1},o_{2})}(CS(o_{1},o_{2}))=RA(o_{1},o_{2}),\qquad i=1,2.
\]

\item The radical axis is the locus of centers of all circles $k$ that
are orthogonal to both $o_{1}$ and $o_{2}$. 
\item \label{enu:CS <-> RA}For any $P\in CS(o_{1},o_{2})$, inversion in
a circle centered at $P$ takes the circle of similitude of the original
circles into the radical axis of the images, and the radical axis
of the original circles into the circle of similitude of the images:
\[
CS(o_{1},o_{2})'=RA(o_{1}',o_{2}'),
\]
\[
RA(o_{1},o_{2})'=CS(o_{1}',o_{2}').
\]
Here $'$ denotes the image of an object under the inversion in a
circle centered at $P\in CS(o_{1},o_{2})$. 
\item Let $K,L,M$ be points on the circles $o_{1},o_{2},CS(o_{1},o_{2})$
respectively. Then
\begin{equation}
\angle AMB=\angle AKB+\angle ALB,\label{eq:angle addition for CS}
\end{equation}
where the angles are taken in the sense of directed angles. 
\item \label{enu:RA and two chords}Let $A_{1}B_{1}$ be a chord of a circle
$k_{1}$ and $A_{2}B_{2}$ be a chord of a circle $k_{2}$. Then $A_{1},B_{1},A_{2},B_{2}$
are on a circle $o$ if and only if $A_{1}B_{1}\cap A_{2}B_{2}\in RA(k_{1},k_{2}).$
\end{enumerate}
It is also useful to recall the construction of the center of a spiral
similarity given the images of two points. Suppose that $A$ and $B$
are transformed into $A'$ and $B'$ respectively. Let $P=AA'\cap BB'$.
The center $O$ of the spiral similarity can be found as the intersection
$O=(ABP)\cap(A'B'P)$, where here and below $(ABP)$ stands for the
circle going through $A,B,P$.

\begin{flushleft}
We will call point $P$ in this construction the \emph{joint point
}associated to two given points $A,B$ and their images $A',B'$ under
spiral similarity. 
\par\end{flushleft}

There is another spiral similarity associated to the same configuration
of points. Let $P'=AB\cap A'B'$ be the joint point for the spiral
similarity taking $A$ and $A'$ into $B$ and $B'$ respectively.
A simple geometric argument shows that the center of this spiral similarity,
determined as the intersection of the circles $(AA'P')\cap(BB'P')$,
coincides with $O$. We will call such a pair of spiral similarities
centered at the same point \emph{associated spiral similarities}. 

Let $H_{i,j}^{W}$ be the spiral similarity centered at $W$ that
takes $o_{i}$ into $o_{j}$. The following Lemma will be useful when
studying properties of the limit point of the iterative process (or
of its inverse):
\begin{lem}
\label{lem:3 pts on CS}\textcolor{black}{Let $o_{1}$ and $o_{2}$
be two circles centered at $O_{1}$ and $O_{2}$ respectively and
intersecting at points $A$ and $B$. Let $W,R,S\in CS(o_{1},o_{2})$
be points on the circle of similitude such that $R$ and $S$ are
symmetric to each other with respect to the line of centers, $O_{1}O_{2}$.
Then the joint points corresponding to taking $O_{1}\to O_{2}$, $R\to R_{1,2}\doteq H_{1,2}^{W}(R)$
by $H_{1,2}^{W}$ and taking $O_{2}\to O_{1}$, $S\to S_{2,1}\doteq H_{2,1}^{W}(S)$
by $H_{2,1}^{W}$ coincide. The common joint point lies on $O_{1}O_{2}$. }
\end{lem}
\begin{figure}
\includegraphics[scale=0.06]{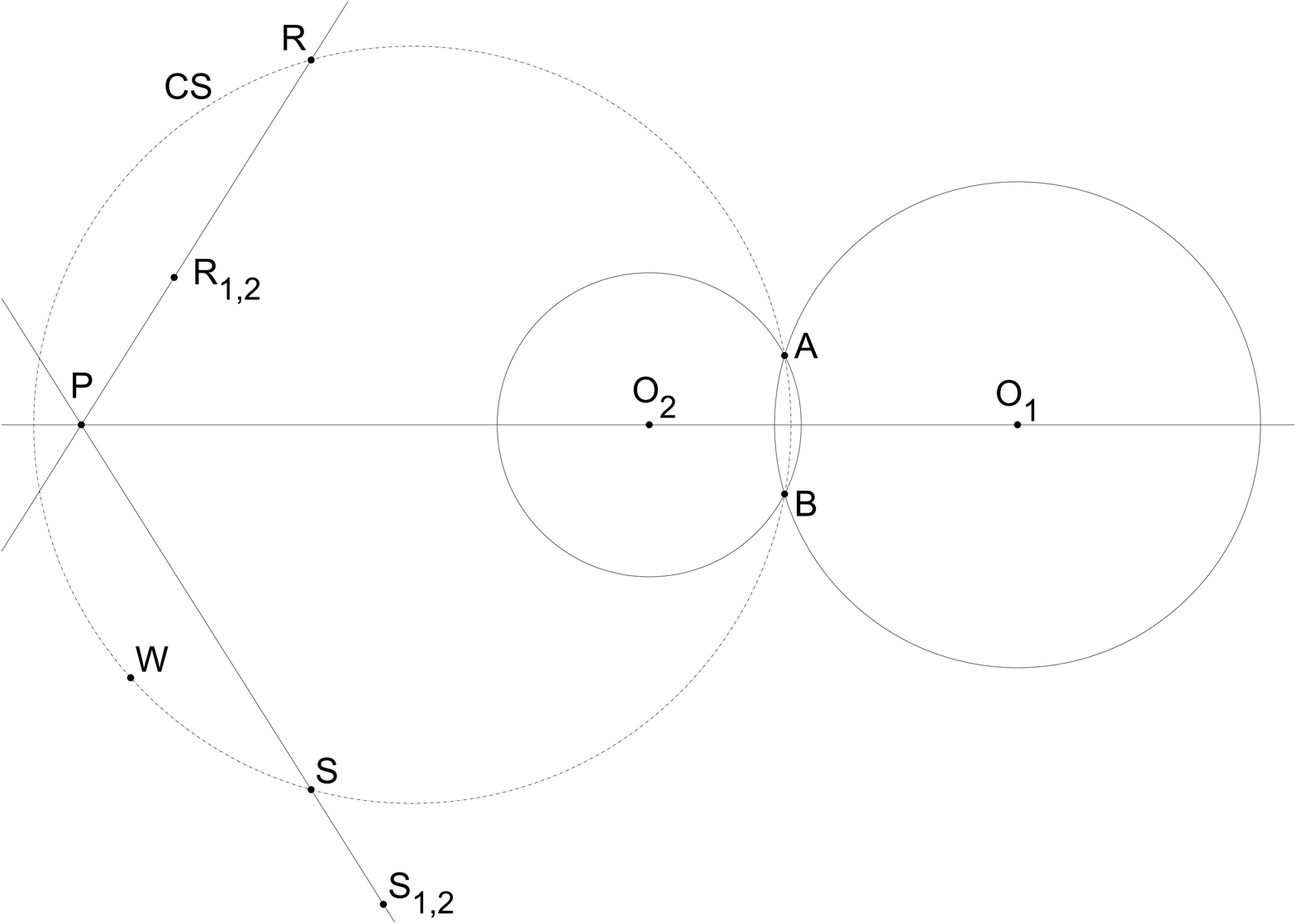}\caption{\label{fig 3 pts on CS}Lemma \ref{lem:3 pts on CS}.}
\end{figure}

\begin{proof}
Perform inversion in the mid-circle. The image of $CS(o_{1},o_{2})$
is the radical axis $RA(o_{1},o_{2})$ (i.e., the line through $A$
and $B$). The images of $R$ and $S$ lie on the line $AB$ and are
symmetric with respect to $I\doteq AB\cap O_{1}O_{2}$. Similarly,
the images of $O_{1}$ and $O_{2}$ are symmetric with respect to
$I$ and lie on the line of centers. By abuse of notation, we will
denote the image of a point under inversion in the mid-circle by the
same letter. 

The lemma is equivalent to the statement that $P\doteq(WO_{1}R)\cap O_{1}O_{2}$
lies on the circle $(WO_{2}S)$. To show this, note that since $P,R,O_{1}$
and W lie on a circle, we have $|IP|\cdot|IO_{1}|=|IW|\cdot|IR|.$
Since $ $$|IO_{2}|=|IO_{1}|$ and $|IR|=|IS|,$ it follows that $|IP|\cdot|IO_{2}|=|IW|\cdot|IS|,$
which implies that $W,P,O_{2},S$ lie on a circle. After inverting
back in the mid-circle, we obtain the result of the lemma. 
\end{proof}
Notice that the lemma is equivalent to the statement that 
\[
RR_{1,2}\cap SS_{2,1}=(WRO_{1})\cap(WSO_{2})\in O_{1}O_{2}.
\]

\subsection{$W$ as the center of spiral similarities for triad circles of $Q^{(1)}$.}

Denote by $o_{1},o_{2},o_{3}$ and $o_{4}$ the triad circles $(D_{1}A_{1}B_{1})$,
$(A_{1}B_{1}C_{1})$, $(B_{1}C_{1}D_{1})$ and $(C_{1}D_{1}A_{1})$
respectively.%
\footnote{In short, the middle vertex defining the circle $o_{i}$ is vertex
number $i$ (the first vertex being $A_{1}$, the second being $B_{1}$,
the third being $C_{1}$ and the last being $D_{1}$). %
} For triad circles in other generations, we add an upper index indicating
 the generation. For example, $o_{1}^{(3)}$ denotes the first triad-circle
in the $3$rd generation quadrilateral, i.e., circle $(D_{3}A_{3}B_{3})$.
Let $T_{1},T_{2},T_{3}$ and $T_{4}$ be the triad triangles $\triangle D_{1}A_{1}B_{1}$,
$\triangle A_{1}B_{1}C_{1}$, $\triangle B_{1}C_{1}D_{1}$ and $\triangle C_{1}D_{1}A_{1}$
respectively.

Consider two of the triad circles of the first generation, $o_{i}$
and $o_{j}$, $i\neq j\in\{1,2,3,4\}$. The set of all possible centers
of spiral similarity taking $o_{i}$ into $o_{j}$ is their circle
of similitude $CS(o_{i},o_{j})$. If $Q^{(1)}$ is a nondegenerate
quadrilateral, it can be shown that $CS(o_{1},o_{2})$ and $CS(o_{1},o_{4})$
intersect at two points (i.e. are not tangent). Let $W$ be the other
point of intersection of $CS(o_{1},o_{2})$ and $CS(o_{1},o_{4})$.%
\footnote{This will turn out to be the same point as the limit point of the
iterative process defined in section \ref{sec: Iterative Process},
so the clash of notation is intentional.%
}

Let $H_{k,l}^{W}$ be the spiral similarity centered at $W$ that
takes $o_{k}$ into $o_{l}$ for any $k,l\in\{1,2,3,4\}$. 
\begin{lem}
\label{lem:Tangency and H12(B)}Spiral similarities $H_{k,l}^{W}$
have the following properties:\end{lem}
\begin{enumerate}
\item $H_{1,2}^{W}(B_{1})=A_{1}\mbox{\ensuremath{\Longleftrightarrow}}H_{2,4}^{W}(A_{1})=C_{1}$.
\item $H_{1,2}^{W}(B_{1})=A_{1}\mbox{\ensuremath{\Longleftrightarrow}}H_{1,4}^{W}(B_{1})=C_{1}$.\end{enumerate}
\begin{proof}
Assume that $H_{1,2}^{W}(B_{1})=A_{1}$. Let $P_{1,2}\doteq A_{1}B_{1}\cap A_{2}B_{2}$
be the joint point of the spiral similarity (centered at $W$) taking
$B_{1}$ into $A_{1}$ and $A_{2}$ into $B_{2}$. Since points $B_{1},P_{1,2},W,A_{2}$
lie on a circle (see Lemma \ref{lem:3 pts on CS}), it follows that
$\angle BWA_{1}=\angle BP_{1,2}A_{2}=\pi/2$. Thus, $A_{2}B_{1}$
is a diameter of $k_{1}\doteq(B_{1}P_{1,2}WA_{2}).$ Since $o_{1}$
is centered at $A_{2}$, the circles $o_{1}$ and $k_{1}$ are tangent
at $B_{1}$. It is easy to see that the converse is also true: if
$o_{1}$ and $(B_{1}WA_{2})$ are tangent at $B_{1}$, then $H_{1,2}^{W}(B_{1})=A_{1}$. 

Since $A_{1},P_{1,2},W,B_{2}$ lie on a circle, it follows that $\angle A_{1}WB_{2}=\angle A_{1}P_{1,2}B_{2}=\pi/2$.
Since $B_{1}\mapsto A_{1}$ and $A_{2}\mapsto B_{2}$ under $H_{1,2}^{W}$,
$\angle B_{1}WA_{2}=\angle A_{1}WB_{2}=\pi/2$. This implies that
the circles $k_{2}\doteq(A_{1}P_{1,2}WB_{2})$ and $o_{2}$ are tangent
at $A_{1}$. It is easy to see that $k_{2}$ is tangent to $o_{2}$
if and only if $k_{1}$ is tangent to $o_{1}$. 

Similarly to the above, let $P_{2,4}\doteq A_{1}H_{2,4}^{W}(A_{1})\cap B_{2}D_{2}$
be the joint point of the spiral similarity centered at $W$ and taking
$o_{2}$ into $o_{4}$. Then $P_{2,4}\in k_{2}$. Similarly to the
argument above, $k_{2}$ is tangent to $o_{2}$ if and only if $k_{4}\doteq(C_{1}P_{2,4}WD_{2})$
is tangent to $o_{4}$. This is equivalent to $H_{2,4}^{W}(A_{1})=C_{1}$. 

The second statement follows since $H_{1,4}^{W}(B_{1})=H_{2,4}^{W}\circ H_{1,2}^{W}(B_{1})=H_{2,4}^{W}(A_{1})=C_{1}$.
(Here and below the compositions of transformations are read right
to left). 
\end{proof}
\begin{figure}
\includegraphics[scale=0.1]{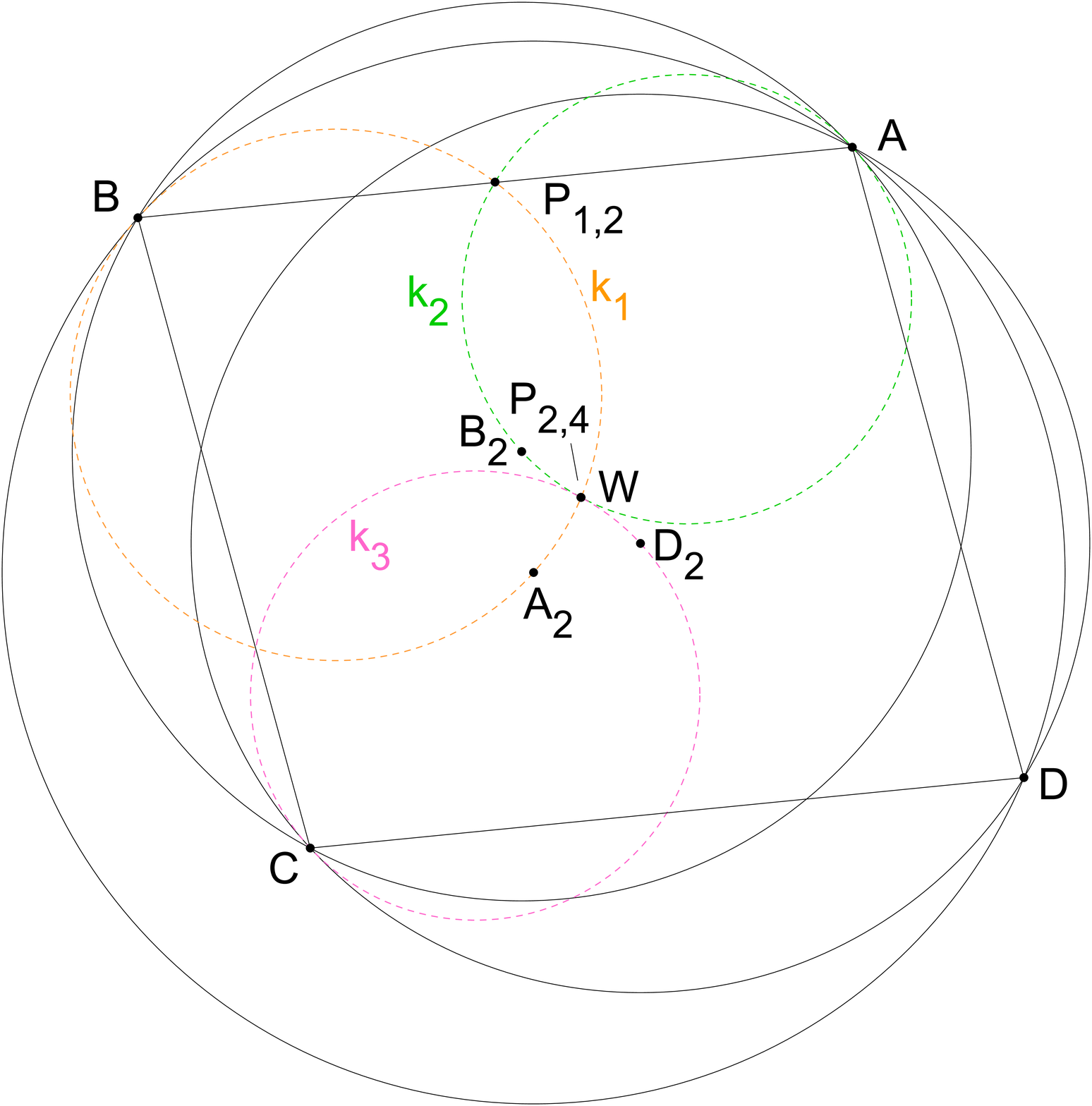}\includegraphics[scale=0.1]{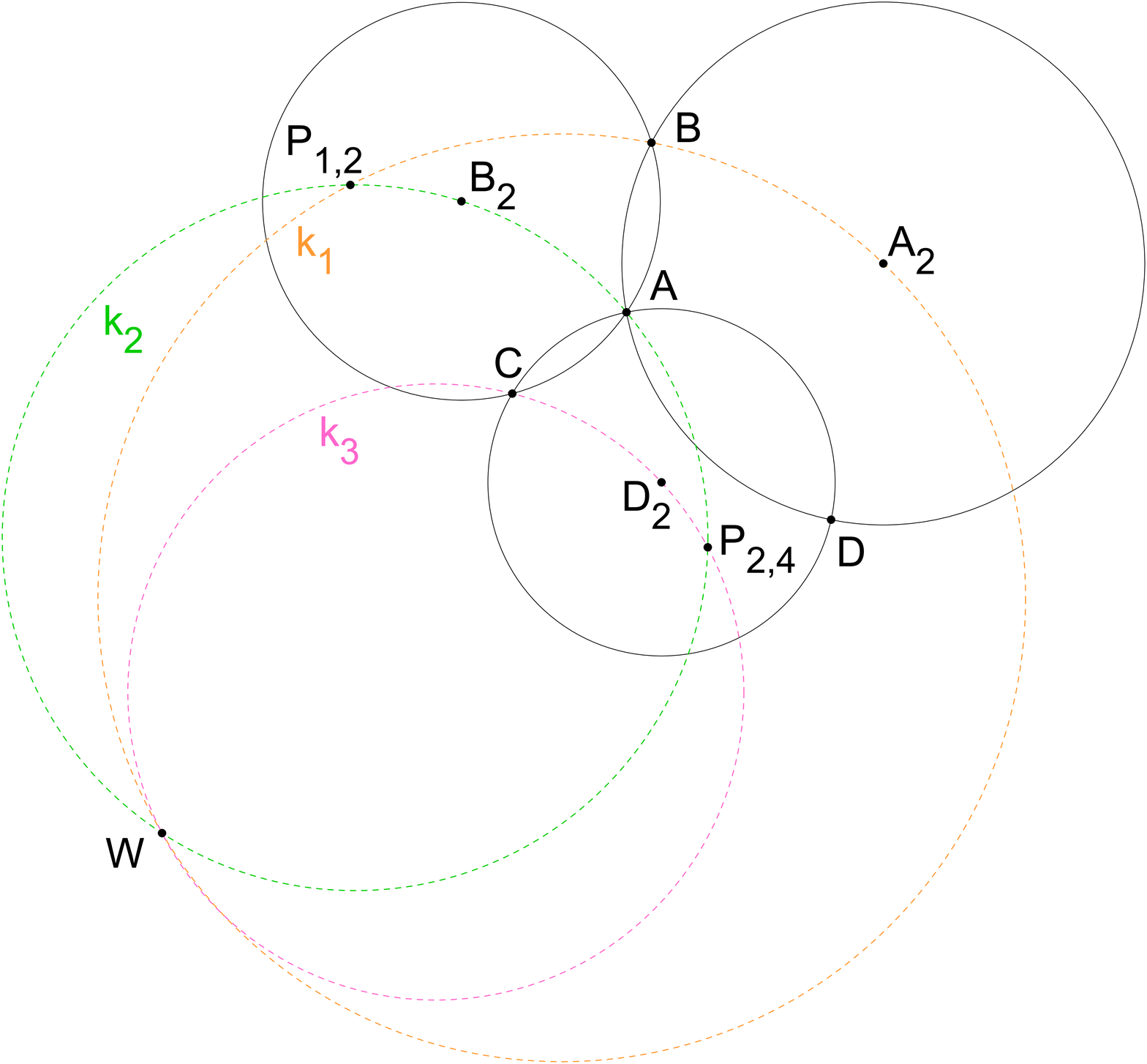}

\caption{Proofs of Lemma \ref{lem:Tangency and H12(B)}\label{fig:H12 (B) is C}
and Lemma \ref{lem:H14 of B is C}.}

\end{figure}

Notice that circles $o_{1}$ and $o_{4}$ have two common vertices,
$A_{1}$ and $D_{1}$. The next Lemma shows that $H_{1,4}^{W}$ takes
$B_{1}$ (i.e., the third vertex on $o_{1}$) to $C_{1}$ (the third
vertex on $o_{4}$). This property is very important for showing that
any triad circle from the first generation can be transformed into
another triad circle from the first generation by a spiral similarity
centered at $W$. Similar properties hold for $H_{1,2}^{W}$ and $H_{2,4}^{W}$.
Namely, we have
\begin{lem}
\label{lem:H14 of B is C} $H_{1,4}^{W}(B_{1})=C_{1}$, $H_{1,2}^{W}(D_{1})=C_{1}$,
$H_{4,2}^{W}(D_{1})=B_{1}$. \end{lem}
\begin{proof}
Lemma \ref{lem:Tangency and H12(B)} shows that $H_{1,2}^{W}(B_{1})=A_{1}$
implies $H_{1,4}^{W}(B_{1})=C_{1}$. Assume that $H_{1,2}^{W}(B_{1})\neq A_{1}$.
To find the image of $B_{1}$ under $H_{1,4}^{W}$, represent the
latter as the composition $H_{2,4}^{W}\circ H_{1,2}^{W}$. First,
$H_{1,2}^{W}(B_{1})=P_{1,2}B_{1}\cap(P_{1,2}B_{2}W)$, where $P_{1,2}$
is as in Lemma \ref{lem:Tangency and H12(B)}, see Figure \ref{fig:H12 (B) is C}.
For brevity, let $B_{1,2}\doteq H_{1,2}^{W}(B_{1})$. (The indices
refer to the fact that $B_{1,2}$ is the image of $B$ under spiral
similarity taking $o_{1}$ into $o_{2}$). 

Now we construct $H_{1,4}^{W}(B_{1})=H_{2,4}^{W}(B_{1,2})$. By Lemma
\ref{lem:3 pts on CS}, $H_{1,4}^{W}(B_{1})=P_{2,4}B_{1,2}\cap(WP_{2,4}D_{2})$,
where $P_{2,4}$ is as in Lemma \ref{lem:Tangency and H12(B)}. Applying
Lemma \ref{lem:3 pts on CS} to the circle $(WP_{2,4}D_{2})$, we
conclude that it passes through $C_{1}$. Since by assumption $H_{1,2}^{W}(B_{1})\neq A_{1}$,
it follows that $H_{2,4}^{W}\circ H_{1,2}^{W}(B_{1})=C_{1}$. Thus,
$H_{1,4}^{W}(B_{1})=C_{1}$. The other statements in the Lemma can
be shown in a similar way. 
\end{proof}
The last Lemma allows us to show that $W$ lies on all of the circles
of similitude $CS(o_{i},o_{j})$.
\begin{thm}
\label{thm:W on CSes}$W\in CS(o_{i},o_{j})$ for all $i,j\in\{1,2,3,4\}$.\end{thm}
\begin{proof}
By definition, $W\in CS(o_{1},o_{2})\cap CS(o_{1},o_{4})\cap CS(o_{2},o_{4})$.
We will show that $W\in CS(o_{3},o_{i})$ for any $i\in\{1,2,4\}$.

Recall that $B_{1}\in CS(o_{1},o_{2})\cap CS(o_{2},o_{3})$. Let $\widetilde{W}$
be the second point in the intersection $CS(o_{1},o_{2})\cap CS(o_{2},o_{3})$,
so that $CS(o_{1},o_{2})\cap CS(o_{2},o_{3})=\{B_{1},\widetilde{W}\}$.
By Lemma \ref{lem:H14 of B is C}, $H_{1,2}^{\widetilde{W}}(D_{1})=C_{1}$.
Since $H_{1,2}^{\widetilde{W}}(A_{2})=B_{2}$, it follows that $H_{1,2}^{\widetilde{W}}=H_{1,2}^{W}$,
which implies that $\widetilde{W}=W$.\textcolor{blue}{{} }Therefore,
$W$ is the common point for all the circles of similitude $CS(o_{i},o_{j})$,
$i,j\in\{1,2,3,4\}$. 
\end{proof}

\subsection{Properties of W}

The angle property (\ref{eq:angle addition for CS}) of the circle
of similitude implies
\begin{cor}
\label{cor:Angle addition for W}The angles subtended by the quadrilateral's
sides at $W$ are as follows (see Figure  \ref{fig:Angles-subtended-by-W}):
\begin{eqnarray*}
\angle AWB & = & \angle ACB+\angle ADB,\\
\angle BWC & = & \angle BAC+\angle BDC,\\
\angle CWD & = & \angle CAD+\angle CBD,\\
\angle DWA & = & \angle DBA+\angle DCA.
\end{eqnarray*}

\end{cor}
\begin{flushleft}
\begin{figure}
\includegraphics[scale=0.3]{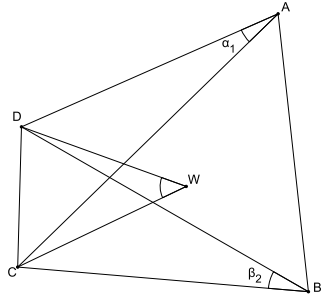}

\caption{\label{fig:Angles-subtended-by-W} $\angle CWD=\angle CAD+\angle CBD.$}
\end{figure}

\par\end{flushleft}

This allows us to view $W$ as a replacement of the circumcenter in
a certain sense: the angle relations above are generalizations of
the relation $\angle AOB=\angle ACB+\angle ADB$ between the angles
in a cyclic quadrilateral $ABCD$ with circumcenter $O$. (Of course,
in this special case, $\angle ACB=\angle ADB$). 

\begin{flushleft}
Since $W\in CS(o_{i},o_{j})$ for all $i,j$, $W$ can be used as
the center of spiral similarity taking any of the triad circles into
another triad circle. This implies the following
\par\end{flushleft}
\begin{thm}
\label{thm: Isoptic property}(Isoptic property) All the triad circles
$o_{i}$ subtend equal angles at $W.$
\end{thm}
\textcolor{black}{In particular, $W$ is inside of all of the triad
circles in the case of a convex quadrilateral and outside of all of
the triad circles in the case of a concave quadrilateral. (This was
pointed out by Scimemi in \cite{Scimemi}). If $W$ is inside of a
triad circle, the isoptic angle equals to $\angle TOT'$, where $T$
and $T'$ are the points on the circle so that $TT'$ goes through
$W$ and $TT'\perp OW$. (See Figure  \ref{fig:Isoptic-angles}, where
$\angle T_{1}A_{2}W$ and $\angle T_{4}B_{2}W$ are halves of the
isoptic angle in $o_{1}$ and $o_{4}$ respectively). If $W$ is outside
of a triad circle centered at $O$ and $WT$ is the tangent line to
the circle, so that $T$ is point of tangency, $\angle OTW$ is half
of the isoptic angle. Inverting in a triad circle of the second generation,
we get that the triad circles are viewed at equal angles from the
vertices opposite to their centers (see Figure \ref{fig:Isoptic-angles}). }

\begin{flushleft}
\begin{figure}
\includegraphics[scale=0.25]{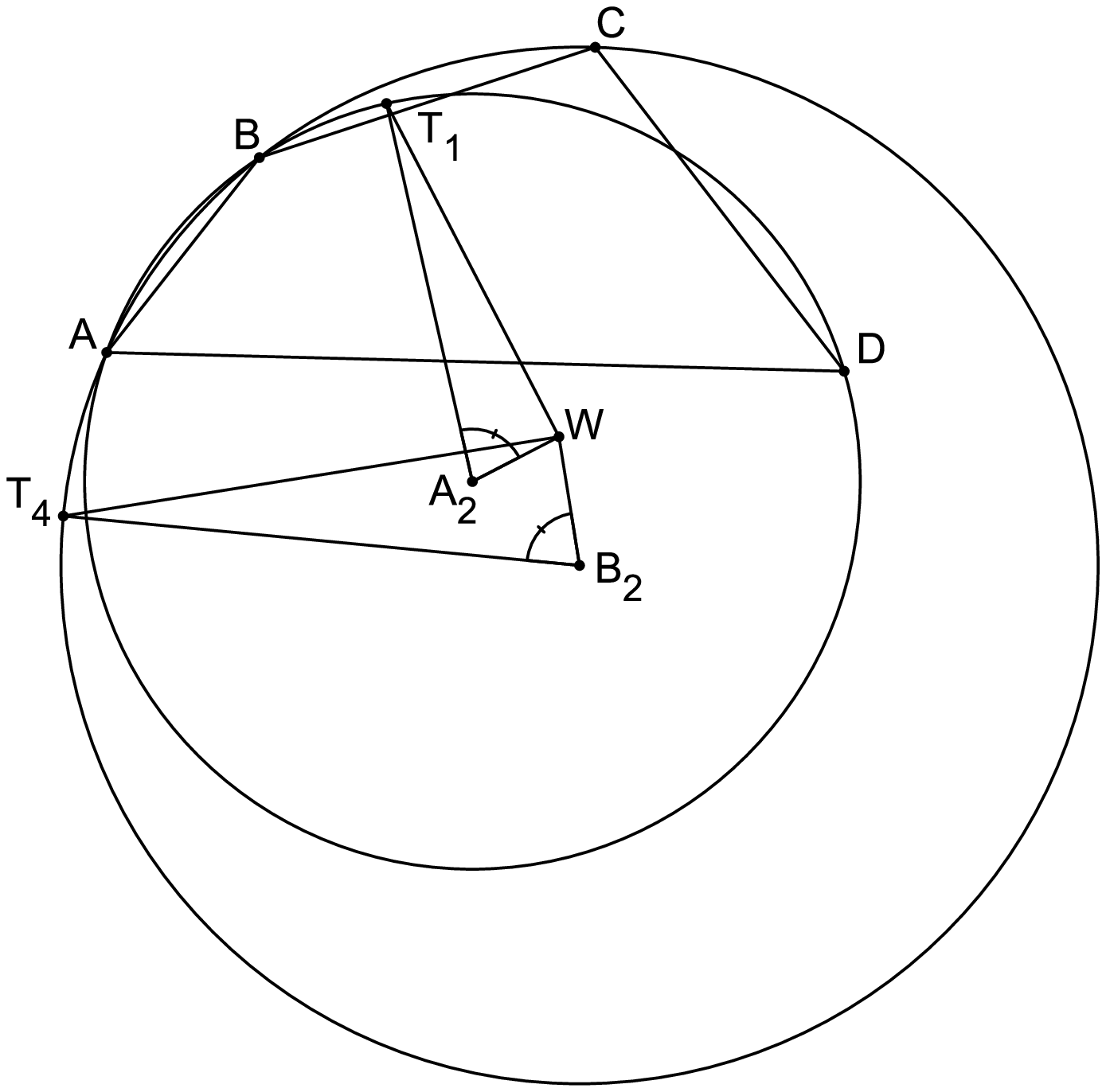}\includegraphics[scale=0.05]{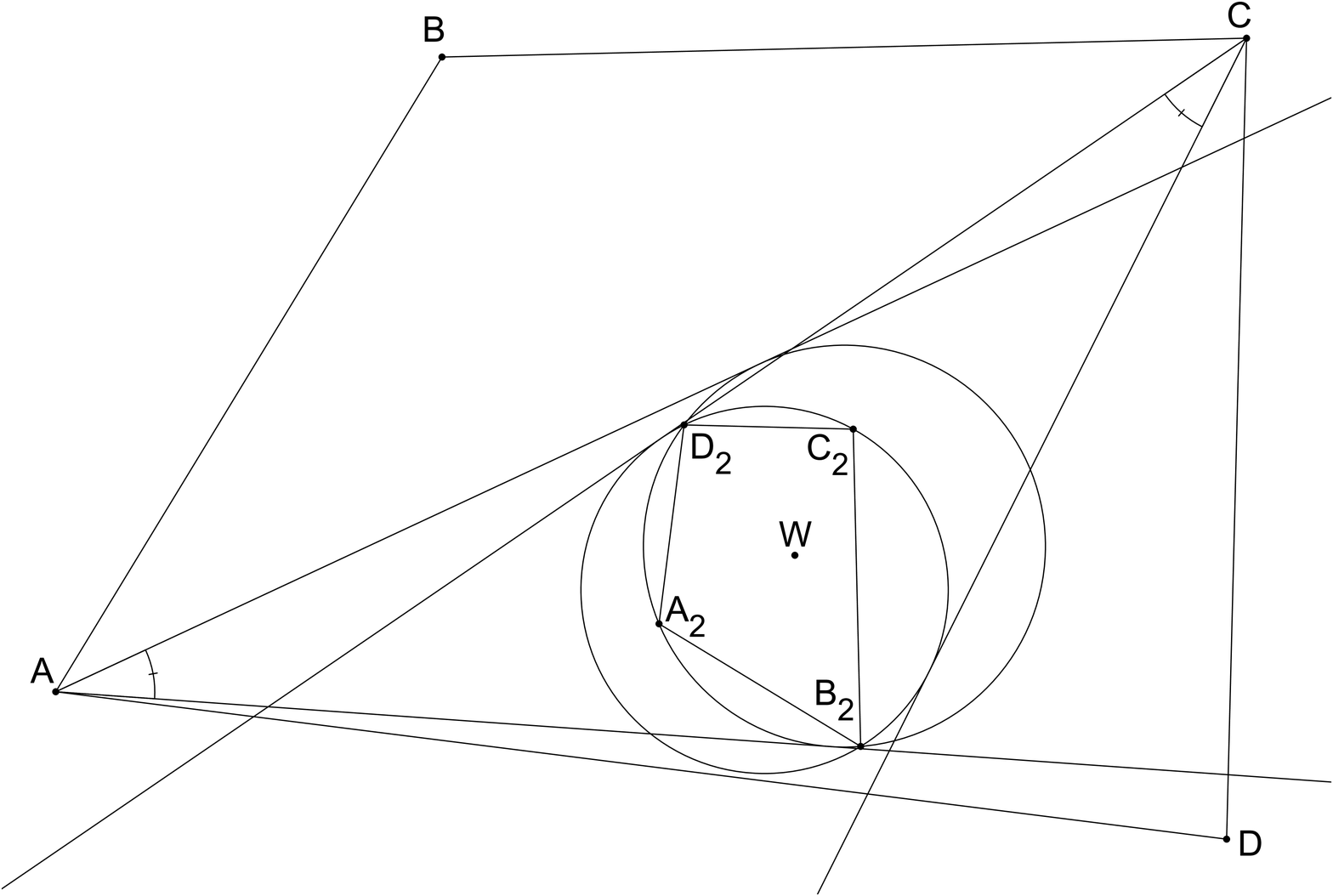}\caption{\label{fig:Isoptic-angles}The isoptic angles before and after inversion. }

\end{figure}

\par\end{flushleft}

\textcolor{black}{Recall that the }\textcolor{black}{\emph{power of
a point}}\textcolor{black}{{} $P$ with respect to a circle $o$ centered
at $O$ with radius $R$ is the square of the length of the tangent
from $P$ to the circle, that is,
\[
h=|PO|^{2}-R^{2}.
\]
The isoptic property implies the following}
\begin{cor}
\textcolor{black}{\label{cor:power of W}The powers of $W$ with respect
to triad circles are proportional to the squares of the radii of the
triad circles.}
\end{cor}
\textcolor{black}{This property of the isoptic point was shown by
Neville in \cite{Neville} using tetracyclic coordinates and the Darboux-Frobenius
identity. }

Let $a,b,c,d$ be sides of the quadrilateral. For any $x\in\{a,b,c,d\}$,
let $F_{x}$ be the foot of the perpendicular bisector of side $x$
on the opposite side. (E.g., $F_{a}$ is the intersection of the perpendicular
bisector to the side $AB$ and the side $CD$). \textcolor{black}{The
following corollary follows from Lemma \ref{lem:H14 of B is C} and
expresses $W$ as the point of intersection of several circles going
through the vertices of the first and second generation quadrilaterals,
as well as the intersections of the perpendicular bisectors of the
original quadrilateral with the opposite sides (see Fig. \ref{fig:W and feet circles}).}

\begin{figure}
\textcolor{black}{\includegraphics[scale=0.2]{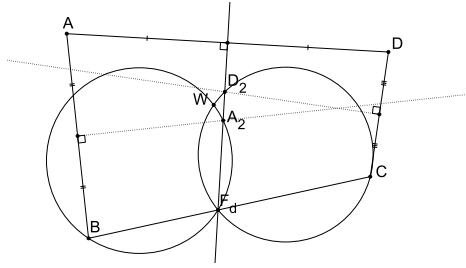}\caption{\textcolor{black}{\label{fig:W and feet circles}$W$ as the intersection
of circles $(B_{1}F_{d}A_{2})$ and $(C_{1}F_{d}D_{2})$ in (\ref{eq:W and feet circles}).}}
}

\end{figure}

\begin{cor}
\label{cor:feetCircles}
\begin{equation}
W=(A_{1}F_{b}B_{2})\cap(A_{1}F_{c}D_{2})\cap(B_{1}F_{c}C_{2})\cap(B_{1}F_{d}A_{2})\cap(C_{1}F_{d}D_{2})\cap(C_{1}F_{a}B_{2})\cap(D_{1}F_{a}A_{2})\cap(D_{1}F_{b}C_{2}),\label{eq:W and feet circles}
\end{equation}
\textup{ This property can be viewed as the generalization of the
following property of the circumcenter of a triangle:}
\end{cor}
\emph{Given a triangle $\triangle ABC$ with sides $a,b,c$ opposite
to vertices $A,B,C$, let $F_{kl}$ denote the feet of the perpendicular
bisector to side $k$ on the side $l$ (or its extension), where $k,l\in\{a,b,c\}$.
Then the circumcenter is the common point of three circles going through
vertices and feet of the perpendicular bisectors in the following
way:}%
\footnote{Note also that this statement is related to Miquel's theorem as follows.
Take any three points $P,Q,R$ on the three circles in (\ref{eq:O and feet circles}),
so that $A,B,C$ are points on the sides $PQ$, $QR$, $PQ$ of $\triangle PQR$.
Then the statement becomes Miquel's theorem for $\triangle PQR$ and
points $A,B,C$ on its sides, with the extra condition that the point
of intersection of the circles $(PAC)$, $(QAB)$, $(RBC)$ is the
circumcenter of $\triangle ABC$.%
}\emph{
\begin{equation}
O=(ABF_{ab}F_{ba})\cap(BCF_{bc}F_{cb})\cap(CAF_{ca}F_{ac}).\label{eq:O and feet circles}
\end{equation}
}

\begin{figure}
\includegraphics[scale=0.2]{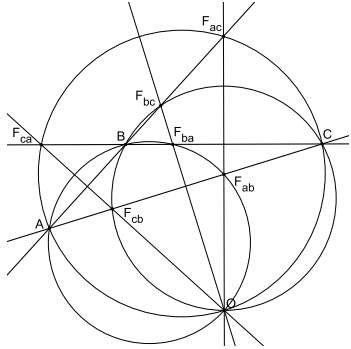}

\caption{Circumcenter as intersection of circles in (\ref{eq:O and feet circles}).}

\end{figure}

The similarity between (\ref{eq:O and feet circles}) and (\ref{eq:W and feet circles})
supports the analogy of the isoptic point with the circumcenter. 

The last corollary provides a quick way of constructing $W$. First,
construct two vertices (e.g., $A_{2}$ and $D_{2}$) of the second
generation by intersecting the perpendicular bisectors. Let $F_{d}$
be the intersection of the lines $A_{2}D_{2}$ and $B_{1}C_{1}$.
Then $W$ is obtained as the second point of intersection of the two
circles $(B_{1}F_{b}A_{2})$ and $(C_{1}F_{b}D_{2})$. 

Recall the definition of isodynamic points of a triangle. Let $\triangle A_{1}A_{2}A_{3}$
be a triangle with sides $a_{1},a_{2},a_{3}$ opposite to the vertices
$A_{1},A_{2},A_{3}$. For each $i,j\in\{1,2,3\}$, where $i\neq j$,
consider the circle $o_{ij}$ centered at $A_{i}$ and going through
$A_{j}$. The circle of similitude $CS(o_{ij},o_{kj})$ of two distinct
circles $o_{ij}$ and $o_{kj}$ is the Apollonian circle with respect
to points $A_{i},A_{k}$ with ratio $r_{ik}=\frac{a_{k}}{a_{i}}$.
It is easy to see that the three Apollonian circles intersect in two
points, $S$ and $S'$, which are called the \emph{isodynamic points
}of the triangle. 

Here are some properties of isodynamic points (see, e.g., \cite{Johnson-book},
\cite{Coxeter} for more details): 
\begin{enumerate}
\item The distances from $S$ (and $S'$) to the vertices are inversely
proportional to the opposite side lengths:
\begin{equation}
|SA_{1}|:|SA_{2}|:|SA_{3}|=\frac{1}{a_{1}}:\frac{1}{a_{2}}:\frac{1}{a_{3}}.\label{eq:isodynamic of triangle}
\end{equation}
 Equivalently,
\[
|SA_{i}|:|SA_{j}|=\sin\alpha_{j}:\sin\alpha_{i},\qquad i\neq j\in\{1,2,3\},
\]
where $\alpha_{i}$ is the angle $\angle A_{i}$ in the triangle.
The isodynamic points can be characterized as the points having this
distance property. Note that since the radii of the circles used to
define the circles of similitude are the sides, the last property
means that distances from isodynamic points to the vertices are inversely
proportional to the radii of the circles. 
\item The pedal triangle of a point on the plane of $\triangle A_{1}A_{2}A_{3}$
is equilateral if and only if the point is one of the isodynamic points.
\item The triangle whose vertices are obtained by inversion of $A_{1},A_{2},A_{3}$
with respect to a circle centered at a point $P$ is equilateral if
and only if $P$ is one of the isodynamic points of $\triangle A_{1}A_{2}A_{3}$. 
\end{enumerate}
\textcolor{black}{It turns out that $W$ has properties (Corollary
\ref{cor:Isodynamic property}, Theorem \ref{thm:Pedal of W}, Theorem
\ref{thm:Varignon Angles}) similar to properties 1--3 of $S$. }
\begin{cor}
\textcolor{black}{(Isodynamic property of $W$) \label{cor:Isodynamic property}The
distances from $W$ to the vertices of the quadrilateral are inversely
proportional to the radii of the triad-circles going through the remaining
three vertices:
\[
|WA_{1}|:|WB_{1}|:|WC_{1}|:|WD_{1}|=\frac{1}{R_{3}}:\frac{1}{R_{4}}:\frac{1}{R_{1}}:\frac{1}{R_{2}},
\]
where $R_{i}$ is the radius of the triad-circle $o_{i}$. Equivalently,
the ratios of the distances from $W$ to the vertices are as follows:
\begin{eqnarray*}
|WA_{1}|:|WB_{1}| & = & |A_{1}C_{1}|\sin\gamma:|B_{1}D_{1}|\sin\delta,\\
|WA_{1}|:|WC_{1}| & = & \sin\gamma:\sin\alpha,\\
|WB_{1}|:|WD_{1}| & = & \sin\delta:\sin\beta.
\end{eqnarray*}
}
\end{cor}
\textcolor{black}{From analysis of similar triangles in the iterative
process, it is easy to see that the limit point of the process satisfies
the above distance relations. Therefore, $W$ (defined at the beginning
of this section as the second point of intersection of $CS(o_{1},o_{2})$
and $CS(o_{1},o_{4})$) is the limit point of the iterative process.}

One more property expresses $W$ as the image of a vertex of the first
generation under the inversion in a triad circle of the second generation.
Namely, we have the following
\begin{thm}
(Inversive property of W) \label{thm:Inversive Property of W}
\begin{equation}
W=\mbox{Inv}_{o_{1}^{(2)}}(A_{1})=\mbox{Inv}_{o_{2}^{(2)}}(B_{1})=\mbox{Inv}_{o_{3}^{(2)}}(C_{1})=\mbox{Inv}_{o_{4}^{(2)}}(D_{1}).\label{eq:inversive property or W}
\end{equation}

\begin{figure}
\includegraphics[scale=0.22]{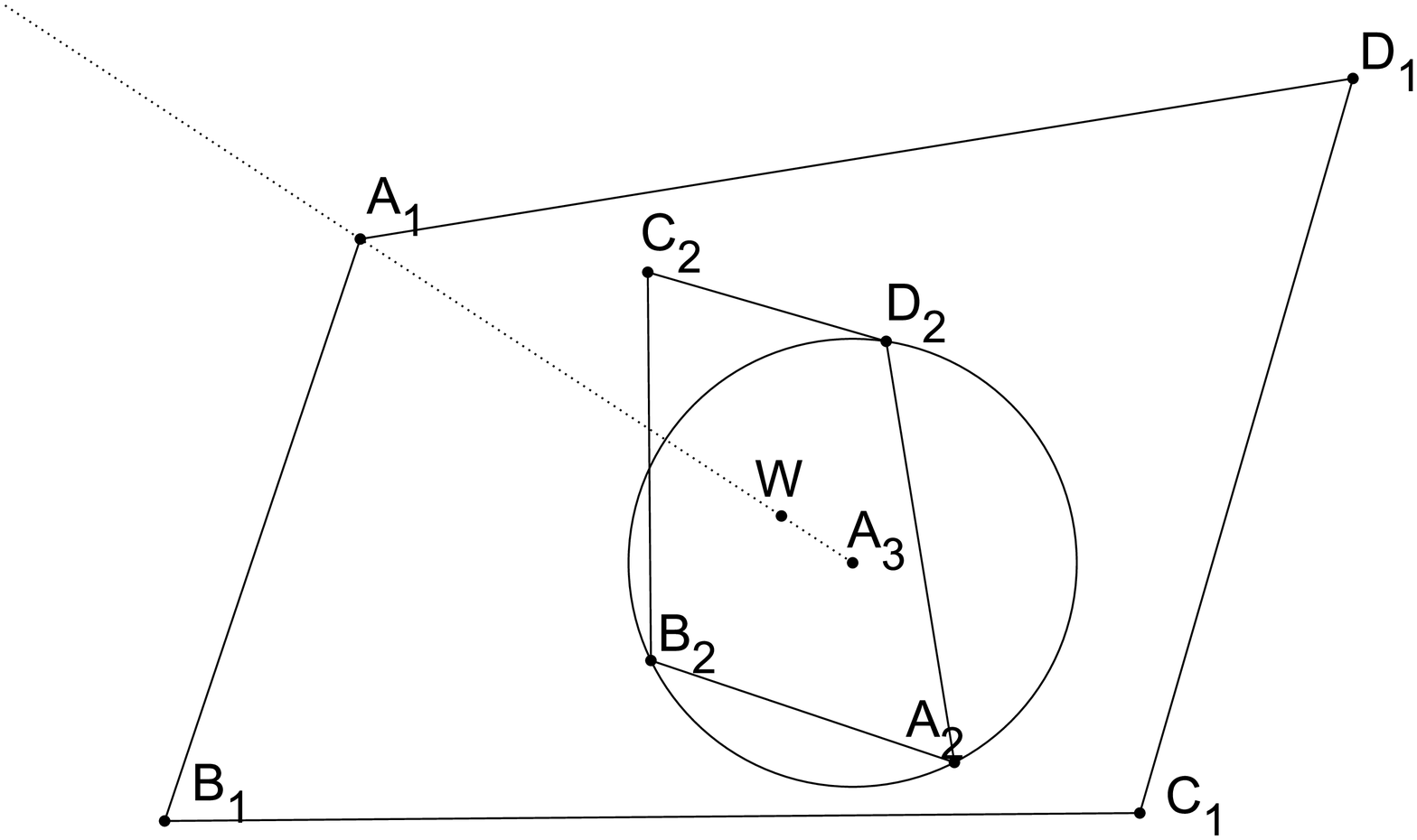}

\caption{Inversive property of $W$}
\end{figure}
\end{thm}
\begin{proof}
To prove the first equality, perform inversion in a circle centered
at $A_{1}$. The image of a point under the inversion will be denoted
by the same letter with a prime. The images of the circles of similitude
$CS(o_{1},o_{2})$, $CS(o_{4},o_{1})$ and $CS(o_{2},o_{4})$ are
the perpendicular bisectors of the segments $A_{2}'B_{2}'$, $D_{2}'A_{2}'$
and $B_{2}'D_{2}'$ respectively. By Theorem \ref{thm:W on CSes},
these perpendicular bisectors intersect in $W'$. Since $W'$ is the
circumcenter of $\triangle D_{2}'A_{2}'B_{2}'$, it follows that $\mbox{Inv}_{o_{1}^{(2)'}}(W')=A_{1}'$.
Inverting back in the same circle centered at $A_{1}$, we obtain
$\mbox{Inv}_{o_{1}^{(2)}}(W)=A_{1}$. The rest of the statements follow
analogously. 
\end{proof}
The fact that the inversions of each of the vertices in triad circles
defined by the remaining three vertices coincide in one point was
proved by Parry and Longuet-Higgins in \cite{Reflections 3}. 

Notice that the statement of Theorem \ref{thm:Inversive Property of W}
can be rephrased in a way that does not refer to the original quadrilateral,
so that we can obtain a property of circumcenters of four triangles
taking a special configuration on the plane. Recall that an inversion
takes a pair of points which are inverses of each other with respect
to a (different) circle into a pair of points which are inverses of
each other with respect to the image of the circle, that is if $S=\mbox{Inv}_{k}(T)$,
then $S'=\mbox{Inv}_{k'}(T')$, where $\prime$ denotes the image
of a point (or a circle) under inversion in a given circle. Using
this and property \ref{enu:CS Inversion exchanges centers} of circles
of similitude, we obtain the corollary below. In the statement, $A,B,C,P,X,Y,Z,O$
play the role of $A_{2}',B_{2}',D_{2}',A_{1},B_{1}',C_{1}',D_{1}',W_{1}'$
in Theorem \ref{thm:Inversive Property of W}. 
\begin{cor}
\label{cor: green pink triangles thm}Let $P$ be a point on the plane
of $\triangle ABC$. Let points $O$, $X$, $Y$ and $Z$ be the circumcenters
of $\triangle ABC$, $\triangle APB$, $\triangle BPC$ $ $ and $\triangle CPA$
respectively. \textup{Then} 
\begin{equation}
\mbox{Inv}_{(ZOX)}(A)=\mbox{Inv}_{(XOY)}(B)=\mbox{Inv}_{(YOZ)}(C)=\mbox{Inv}_{(XYZ)}(P).\label{eq:green pink triangles}
\end{equation}
Furthermore,
\[
Iso{}_{\triangle ZOX}(A)=Y,\quad Iso{}_{\triangle XOY}(B)=Z,\quad Iso_{\triangle YOZ}(C)=X.
\]

\begin{figure}[H]
\includegraphics[scale=0.05]{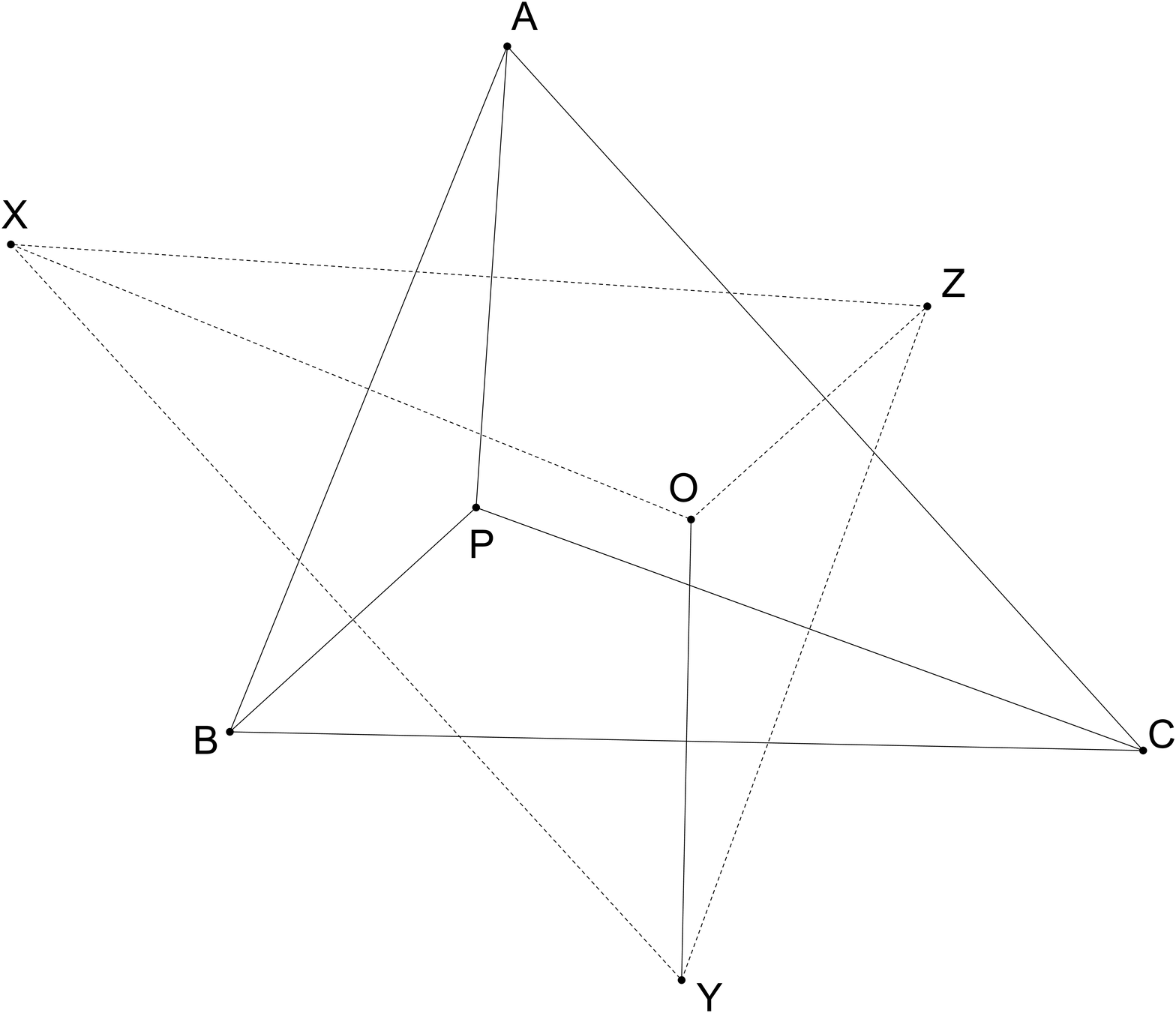}

\caption{Theorem \ref{cor: green pink triangles thm}.}
\end{figure}

\end{cor}
\textcolor{black}{Combining the description of the reverse iterative
process (Theorem \ref{thm:Reversing using Isogonal Conjugation})
}and the inversive property of $W$ (Theorem \ref{thm:Inversive Property of W}),
we obtain one more direct way of constructing $W$ without having
to refer to the iterative process:
\begin{thm}
\label{thm:Inverse-Isogonal System}Let $A,B,C,D$ be four points
in general position. Then
\[
W=\mbox{Inv}_{o_{3}}\circ Iso_{T_{3}}(A_{1})=\mbox{Inv}_{o_{4}}\circ Iso_{T_{4}}(B_{1})=\mbox{Inv}_{o_{1}}\circ Iso_{T_{1}}(C_{1})=\mbox{Inv}_{o_{2}}\circ Iso_{T_{2}}(D_{1}),
\]
where $o_{i}$ is the $i$\textup{th} triad circle, and $T_{i}$ is
the $i$\textup{th} triad triangle. 
\end{thm}
\textcolor{black}{This property suggests a surprising relation between
inversion and isogonal conjugation.}

Taking into account that the circumcenter and the orthocenter of a
triangle are isogonal conjugates of each other, we obtain the following
\begin{cor}
$W$ is the point at infinity if and only if the vertices of the quadrilateral
form an orthocentric system.
\end{cor}

\subsection{W as the center of similarity for any pair of triad circles}

To show that $W$ is the center of spiral similarity for any pair
of triad circles (of possibly different generations), we first need
to prove Lemmas \ref{lem:Variable Circle Lemma}---\ref{lem:Chord Spiral}
below. 

The following lemma shows that given three points on a circle ---
two fixed and one variable --- the locus of the joint points of the
spiral similarities taking one fixed point into the other applied
to the variable point is a line.
\begin{lem}
\label{lem:Variable Circle Lemma}Let $M,N\in o$ and $W\notin o$.
For every point $L\in o$, define 
\[
J\doteq(MWL)\cap NL.
\]
The locus of points $J$ is a straight line going through $W$. \end{lem}
\begin{proof}
For each point $L\in o$, let $K$ be the center of the circle $k\doteq(MWL)$.
The locus of centers of the circles $k$ is the perpendicular bisector
of the segment $MW$. Since $M\in o\cap k$, there is a spiral similarity
centered at $M$ with joint point $L$ that takes $k$ into $o$.
This spiral similarity takes $K\mapsto O$ and $J\mapsto N$, where
$O$ is the center of $o$. Thus, $\triangle MOK\simeq\triangle MNJ$.
Since $M,O,K$ are fixed and the locus of $K$ is a line (the perpendicular
bisector), the locus of points $J$ is also a line. 

To show that the line goes through $W$, let $L=NW\cap o$. Then $J=W$. 
\end{proof}
\begin{figure}
\includegraphics[scale=0.07]{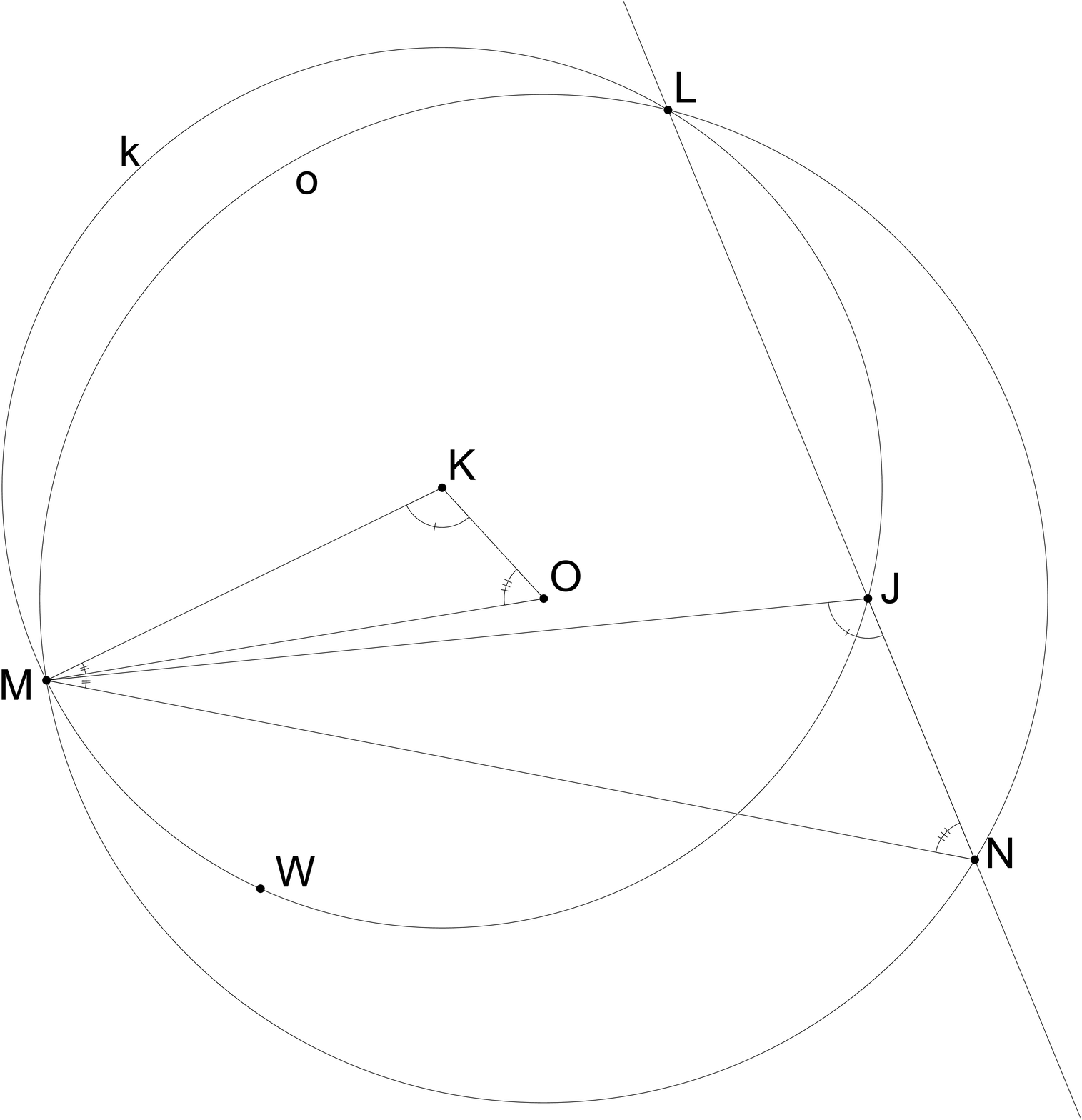}\includegraphics[scale=0.17]{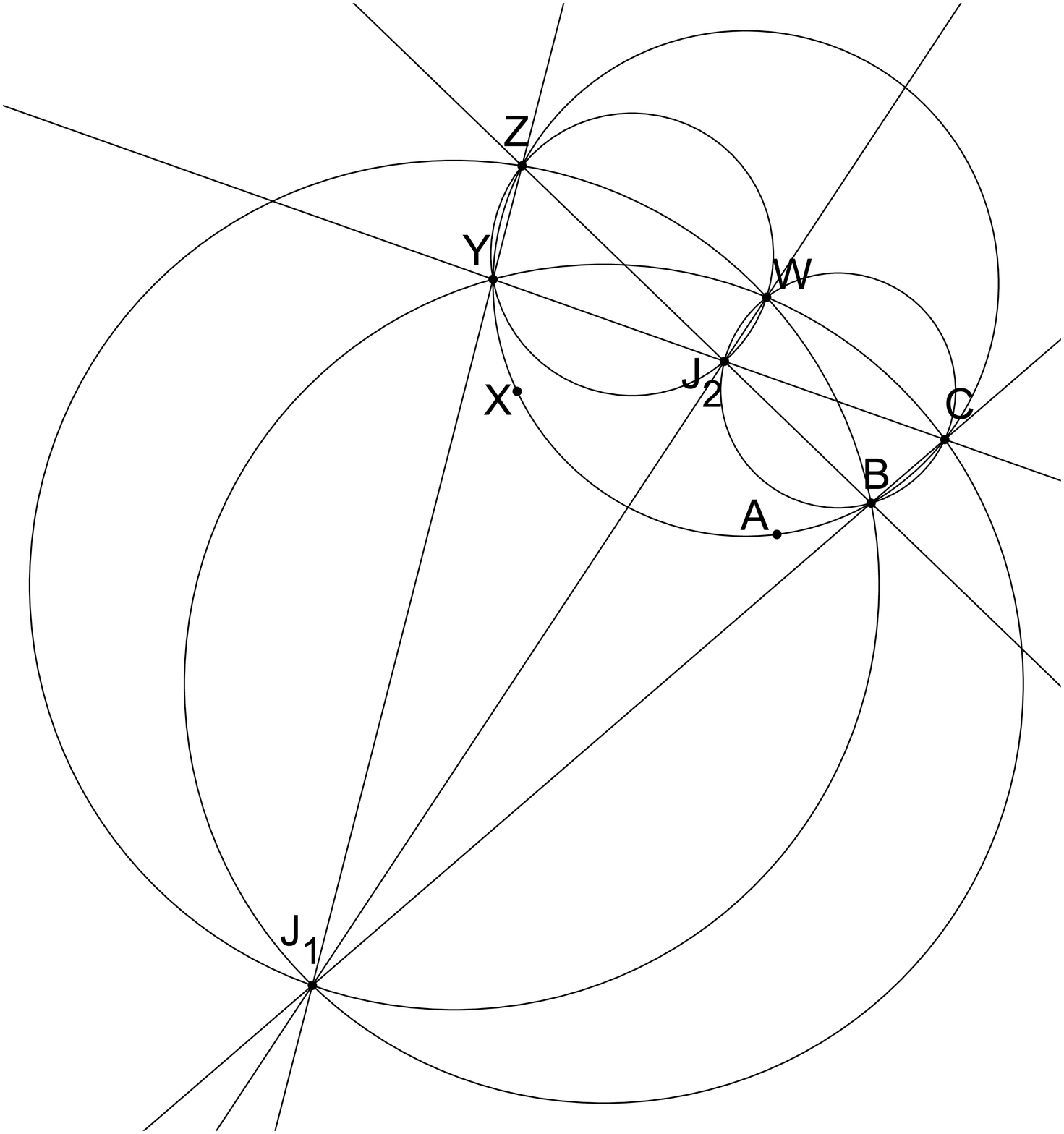}

\caption{Lemma \ref{lem:Variable Circle Lemma} and Lemma \ref{thm:Cross-Spiral}.}
\end{figure}
In the setup of the lemma above, let $H_{L,N}^{W}$ be the spiral
similarity centered at $W$ that takes $L$ into $N$. Let $M'$ be
the image of $M$ under this spiral similarity. Then $J$ is the joint
point for the spiral similarity taking $L\mapsto N$ and $M\mapsto M'$.

The following two results are used for proving that $W$ lies on the
circle of similitude of $o_{3}$ and $o_{1}^{(2)}$. 
\begin{lem}
\label{thm:Cross-Spiral} Let $AC$, $ZX$ be two distinct chords
of a circle $o$, and $W$ be the center of spiral similarity taking
$ZX$ into $AC$. Let $H_{B,C}^{W}$ be the spiral similarity centered
at $W$ that takes a point $B\in o$ into $C$. Then $H_{B,C}^{W}(Z)\in o$. \end{lem}
\begin{proof}
Let $l$ be the locus of the joint points corresponding to $M=Z$,
$N=C$ in Lemma \ref{lem:Variable Circle Lemma}. Let $J_{1}$ be
the joint point corresponding to $L=B$. Then $J_{1},W\in l$. 

Let $J_{2}$ be the joint point corresponding to $M=C$, $N=Z$ and
$L=B$ in Lemma \ref{lem:Variable Circle Lemma}. 

Let $Y=J_{2}C\cap J_{1}Z$. By properties of spiral similarity, $Y=H_{B,C}^{W}(Z)$. 

Notice that by definition of $J_{1}$, points $J_{1},B,C$ are on
a line. Similarly, by definition of $J_{2}$, points $J_{2},B,Z$
are on a line as well. By definition of $Y$, points $Y,J_{2},C$
are on a line, as are points $Z,Y,J_{1}$. The intersections of these
four lines form a complete quadrilateral. By Miquel's theorem, the
circumcircles of the triangles $\triangle BJ_{1}Z,\ \triangle BJ_{2}C,\ \triangle J_{2}YZ,\ \triangle CJ_{1}Y$
have a common point, the Miquel point for the complete quadrilateral.
By definitions of $J_{1}$ and $J_{2}$ , $(BJ_{2}C)\cap(BZJ_{1})=\{B,W\}$.
Thus, the Miquel point is either $B$ or $W$. It is easy to see that
$B$ can not be the Miquel point (if $B\neq C,Z$). Thus, $W$ is
the Miquel point of the complete quadrilateral. This implies that
$(YCJ_{1})$, $(YZJ_{2})$ both go through $W$. 

Consider the circles $k_{1}=(ZWJ_{2}Y)$ and $k_{2}=(CWJ_{2}B)$.
Then $RA(k_{1},k_{2})=l$. Since $ZY\cap BC=J_{1}\in l=RA(k_{1},k_{2})$,
by property \ref{enu:RA and two chords} in section \ref{sub:CS, MS, RA.},
points $Z,Y,B,C$ are on a circle. Thus, $Y\in o$. \end{proof}
\begin{rem*}
Notice that in the proof of the Lemma above there are three spiral
similarities centered at $W$ that take each of the sides of $\triangle XYZ$
into the corresponding side of $\triangle CBA$. We will call such
a construction a \emph{cross-spiral }and say that the two triangles
are obtained from each other via a cross-spiral. %
\footnote{Clearly, the sides of any triangle can be taken into the sides of
any another triangle by three spiral similarities. The special property
of the cross-spiral is that the centers of all three spiral similarities
are at the same point. %
}\end{rem*}
\begin{lem}
\label{lem:Chord Spiral} Let $PQ$ be a chord on a circle $o$ centered
at $O$. If $W\notin(POQ)$, there is a spiral similarity centered
at $W$ that takes $PQ$ into another chord of the circle $o$. \end{lem}
\begin{proof}
Let $H_{P,P'}^{W}$ be the spiral similarity centered at $W$ that
takes $P$ into another point $P'$ on circle $o$. As $P'$ traces
out $o$, the images $H_{P,P'}^{W}(Q)$ of $Q$ trace out another
circle, $o_{Q}$. To see this, consider the associated spiral similarity
and notice that $H_{P,Q}^{W}(P')=Q'$. Since $P'$ traces out $o$,
$H_{P,Q}^{W}(o)=o_{Q}$. Since $Q=H_{P,P}^{W}(Q)\in o_{Q}$, it follows
that $Q\in o\cap o_{Q}$.

Suppose that $o$ and $o_{Q}$ are tangent at $Q$. From $H_{P,Q}(o)=o_{Q}$
it follows that the joint point is $Q$, and therefore the quadrilateral
$PQWO$ must be cyclic. Since $W\notin(POQ)$, this can not be the
case. Thus, the intersection $o\cap o_{Q}$ contains two points, $Q$
and $Q'$. This implies that there is a unique chord, $P'Q'$, of
$o$ to which $PQ$ can be taken by a spiral similarity centered at
$W$. 
\end{proof}
\begin{figure}
\includegraphics[scale=0.45]{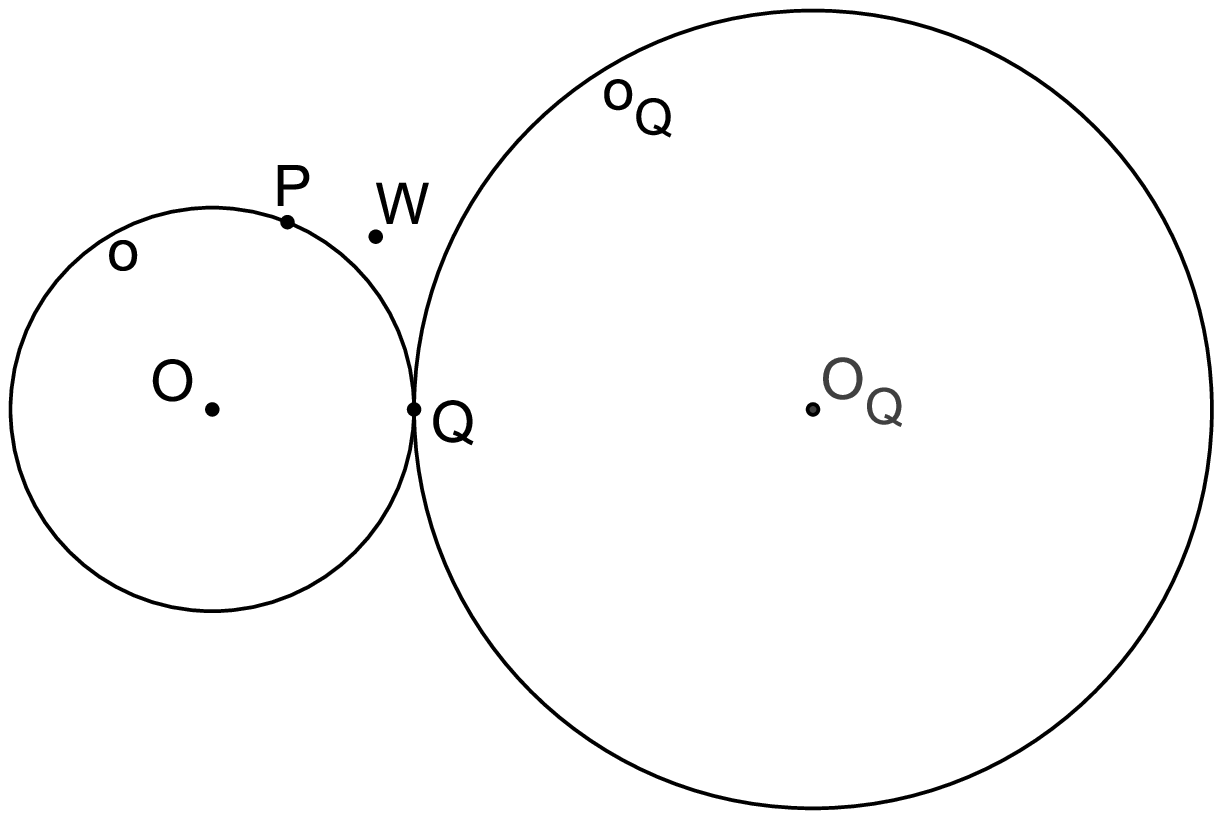}

\caption{Proof of Lemma \ref{lem:Chord Spiral}.}

\end{figure}

\begin{thm}
\label{thm:W on CS 1-2}$W\in CS(o_{3},o_{1}^{(2)})$. \end{thm}
\begin{proof}
We've shown previously that $W$ is on all six circles of similitude
of $A_{1}B_{1}C_{1}D_{1}$. Since $W$ has the property that 
\begin{eqnarray*}
H_{C_{1},B_{2}}^{W}:\ C_{1}\mapsto B_{2},D_{1}\mapsto A_{2},\\
H_{B_{1},A_{2}}^{W}:\ B_{1}\mapsto A_{2},C_{1}\mapsto D_{2},\\
\end{eqnarray*}
it follows that 
\[
H_{B_{2},C_{1}}^{W}H_{B_{1},A_{2}}^{W}(B_{1})=H_{B_{2},C_{1}}^{W}(A_{2})=D_{1}.
\]

Since the spiral similarities centered at $W$ commute, it follows
that
\[
H_{B_{2},C_{1}}^{W}H_{B_{1},A_{2}}^{W}(B_{2})=H_{B_{1},A_{2}}^{W}H_{B_{2},C_{1}}^{W}(B_{2})=H_{B_{1}A_{2}}^{W}(C_{1})=D_{2}.
\]
 This means that there is a spiral similarity centered at $W$ that
takes $B_{1}D_{1}$ into $B_{2}D_{2}$. Therefore, $\triangle B_{1}C_{1}D_{1}$
and $\triangle D_{2}A_{2}B_{2}$ are related by a cross-spiral centered
at $W$. 

We now show that there is a cross-spiral that takes $\triangle D_{2}A_{2}B_{2}$
into another triangle, $\triangle XYZ$, with vertices on the same
circle, $o_{1}^{(2)}=(D_{2}A_{2}B_{2})$. This will imply that there
is a spiral similarity centered at $W$ that takes $\triangle B_{1}C_{1}D_{1}$
into  $\triangle XYZ$. This, in turn, implies that $W$ is a center
of spiral similarity taking $o_{3}$ into $o_{1}^{(2)}$. 

Assume that $W\in(B_{2}A_{3}D_{2})$. Since inversion in $(D_{2}A_{2}B_{2})$
takes $W$ into $A_{1}$ and $(B_{2}A_{3}D_{2})$ into $B_{2}D_{2}$,
it follows that $A_{1}\in B_{2}D_{2}$. This can not be the case for
a nondegenerate quadrilateral. Thus, $W\in(B_{2}A_{3}D_{2})$. 

By Lemma \ref{lem:Chord Spiral}, there is a spiral similarity centered
at $W$ that takes the chord $B_{2}D_{2}$ into another chord, $XZ$,
of the circle $(D_{2}A_{2}B_{2})$. Thus, there is a spiral similarity
taking $B_{2}D_{2}$ into $XZ$ and centered at $W$. 

By Lemma \ref{thm:Cross-Spiral}, there is a point $Y\in o_{1}^{(2)}$
such that $\triangle XYZ$ and $\triangle B_{2}A_{2}D_{2}$ are related
by a cross-spiral centered at $W$. (See also the remark after Lemma
\ref{thm:Cross-Spiral}).

By composing the two cross-spirals, we conclude that $\triangle XYZ\sim\triangle D_{1}C_{1}B_{1}$.
Since $(XYZ)=o_{1}^{(2)}$ and $(D_{1}C_{1}B_{1})=o_{3}$, it follows
that $W\in CS(o_{1}^{(2)},o_{3})$. \end{proof}
\begin{cor}
$W\in CS(o_{i}^{(1)},o_{j}^{(k)})$ for any $i,j$, k. \end{cor}
\begin{proof}
Since there is a spiral similarity centered at $W$ that takes $A_{1}B_{1}$
into $C_{2}D_{2}$, Theorem \ref{thm:W on CS 1-2} implies that $W\in CS(o_{1},o_{4}^{(2)})$.
Since $W\in CS(o_{1},o_{2})$, it follows that $W\in CS(o_{4}^{(2)},o_{2})$.
Since $W$ is on two circles of similitude for the second generation,
it follows that it is on all four. Furthermore, we can apply Theorem
\ref{thm:W on CS 1-2} to the triad circles of the second and third
generation to show that $W$ is also on all four circles of similitude
of the third generation. 

Finally, a simple induction argument shows that $W\in CS(o_{j}^{(1)},o_{i}^{(k)})$.
Assuming $W\in CS(o_{j}^{(1)},o_{i}^{(k-1)})$, Theorem \ref{thm:W on CS 1-2}
implies that $W\in CS(o_{i}^{(k-1)},o_{i}^{(k)})$. Thus, $W\in CS(o_{j}^{(1)},o_{i}^{(k)})$. 
\end{proof}
Using this, we can show that $W$ lies on all the circles of similitude:
\begin{thm}
\label{thm: W on all CSs} $W\in CS(o_{i}^{(k)},o_{j}^{(l)})$ for
all $i,j\in\{1,2,3,4\}$ and any $k,l$. 
\end{thm}
Recall that the \textit{complete quadrangle }is the configuration
of $6$ lines going through all possible pairs of $4$ given vertices. 
\begin{thm}
\label{thm:(Self-Duality-under-Inversion)}(Inversion in a circle
centered at $W$) Consider the complete quadrangle determined by a
nondegenerate quadrilateral. Inversion in $W$ transforms
\begin{itemize}
\item $6$ lines of the complete quadrilateral into the $6$ circles of
similitude of the triad circles of the image quadrilateral;
\item $6$ circles of similitude of the triad circles into the $6$ lines
of the image quadrangle. 
\end{itemize}
\end{thm}
\begin{proof}
Observe that the $6$ lines of the quadrangle are the radical axes
of the triad circles taken in pairs. Since $W$ belongs to all the
circles of similitude of triad circles, by property \ref{enu:CS <-> RA}
in section \ref{sub:CS, MS, RA.}, inversion in a circle centered
in $W$ takes radical axes into the circles of similitude. This implies
the statement. 
\end{proof}

\section{Pedal properties\label{sec:Pedal}}

\subsection{Pedal of $W$ with respect to the original quadrilateral}

Since $W$ has a distance property similar to that of the isodynamic
points of a triangle (see Corollary \ref{cor:Isodynamic property}),
it is interesting to investigate whether the analogy between these
two points extends to pedal properties. In this section we show that
the pedal quadrilateral of $W$ with respect to $A_{1}B_{1}C_{1}D_{1}$
(and, more generally, with respect to any $Q^{(n)}$) is a nondegenerate
parallelogram. Moreover, $W$ is the unique point whose pedal has
such a property. These statements rely on the fact that $W$ lies
on the intersection of two circles of similitude, $CS(o_{1},o_{3})$
and $CS(o_{2},o_{4})$. 

First, consider the pedal of a point that lies on one of these circles
of similitude. 
\begin{lem}
\label{lem:pedal of pt on CS}Let $P_{a}P_{b}P_{c}P_{d}$ be the pedal
quadrilateral of $P$ with respect to $ABCD_{1}$. Then
\begin{itemize}
\item $P_{a}P_{b}P_{c}P_{d}$ is a trapezoid with $P_{a}P_{d}||P_{b}P_{c}$
if and only if $P\in CS(o_{2},o_{4})$;
\item $P_{a}P_{b}P_{c}P_{d}$ is a trapezoid with $P_{a}P_{b}||P_{c}P_{d}$
if and only if $P\in CS(o_{1},o_{3})$. 
\end{itemize}
\end{lem}
\begin{proof}
Assume that $P\in CS(o_{2},o_{4})$. Let $K=AC\cap P_{a}P_{d}$ and
$L=AC\cap P_{b}P_{c}$. We will show that $\angle AKP_{d}+\angle CLP_{c}=\pi$,
which implies $P_{a}P_{d}||P_{b}P_{c}$. 

Let $\theta=\angle APP_{a}$. Since $AP_{a}PP_{d}$ is cyclic, $\angle AP_{d}P_{a}=\theta$.
Then
\begin{equation}
\angle AKP_{d}=\pi-\alpha_{1}-\theta.\label{eq:angle 1}
\end{equation}
On the other hand, $\angle CLP_{c}=\pi-\gamma_{2}-\angle LP_{c}C$.
Since $PP_{b}CP_{c}$ is cyclic, it follows that $\angle LP_{c}C=\angle P_{b}PC.$ 

We now find the latter angle. Since $P\in CS(o_{2},o_{4})$, by a
property of circle of similitude (see (\ref{eq:angle addition for CS})),
it follows that $\angle APC=\pi+\delta+\beta$. Since $P_{a}PP_{b}B$
is cyclic, $\angle P_{a}PP_{b}=\pi-\beta$. Therefore, $\angle P_{b}PC=\delta-\theta.$
This implies that
\begin{equation}
\angle CLP_{c}=\pi-\gamma_{2}-\delta+\theta.\label{eq:angle 2}
\end{equation}
Adding (\ref{eq:angle 1}) and (\ref{eq:angle 2}), we obtain $ $$\angle AKP_{d}+\angle CLP_{c}=\pi$. 

Reasoning backwards, it is easy to see that $P_{a}P_{d}||P_{b}P_{c}$
implies that $P\in CS(o_{2},o_{4}).$
\end{proof}
\begin{figure}

\includegraphics[scale=0.35]{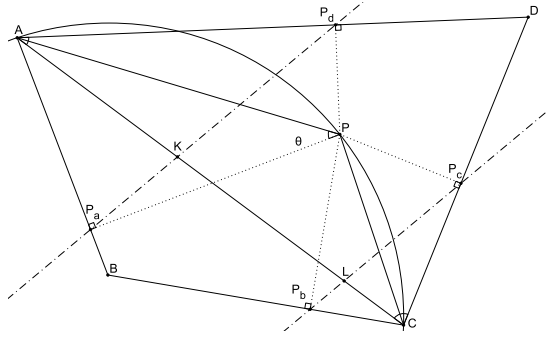}\caption{The pedal quadrilateral of a point on $CS(o_{2},o_{4})$ has two parallel
sides. }

\end{figure}

Let $S$ be the second point of intersection of $CS(o_{1},o_{3})$
and $CS(o_{2},o_{4})$, so that $CS(o_{1},o_{3})\cap CS(o_{2},o_{4})=\{W,S\}$.
The Lemma above implies that the pedal quadrilateral of a point is
a parallelogram if and only if this point is either $W$ or $S$. 
\begin{thm}
\label{thm:Varignon Angles}The pedal quadrilateral of $W$ is a parallelogram
whose angles equal to those of the Varignon parallelogram.\end{thm}
\begin{proof}
Since $W\in CS(o_{1},o_{2})\cap CS(o_{3},o_{4})$, property (\ref{eq:angle addition for CS})
of the circle of similitude implies that
\begin{eqnarray*}
\angle AWB & = & \angle ACB+\angle ADB=\gamma_{1}+\delta_{2},\\
\angle CWD & = & \angle CAD+\angle CBD=\alpha_{1}+\beta_{2},
\end{eqnarray*}
where $\alpha_{i},\beta_{i},\gamma_{i},\delta_{i}$ are the angles
between the quadrilateral's sides and diagonals, as before (see Figure
 \ref{fig:angles between sides diagonals}). Let $\angle AWW_{a}=x$
and $\angle W_{c}WC=y$. Since the quadrilaterals $W_{a}WW_{d}A$
and $W_{c}WW_{b}C$ are cyclic, $\angle W_{a}W_{d}A=x$ and $\angle W_{c}W_{b}C=y$.
Therefore, 
\[
\angle W_{a}W_{b}B=\angle AWB-\angle AWW_{a}=\gamma_{1}+\delta_{2}-x,
\]
\[
\angle W_{c}W_{d}D=\angle CWD-\angle W_{c}WC=\alpha_{1}+\beta_{2}-y.
\]

Finding supplements and adding, we obtain

\begin{eqnarray*}
\angle W_{a}W_{d}W_{c}+\angle W_{a}W_{b}W_{c} & = & (\pi-x-\alpha_{1}-\beta_{2}+y)+(\pi-y-\gamma_{1}-\delta_{2}+x)=\\
=2\pi-\alpha_{1}-\beta_{2}-\gamma_{1}-\delta_{2} & = & 2\pi-(2\pi-2\angle AIC)=2\angle AIC,
\end{eqnarray*}
where $\angle AIC$ is the angle formed by the intersection of the
diagonals. Thus, $W_{a}W_{b}W_{c}W_{d}$ is a parallelogram with the
same angles as those of the Varignon parallelogram $M_{a}M_{B}M_{b}M_{c}$,
where $M_{x}$ is the midpoint of side $x$, for any $x\in\{a,b,c,d\}$.
\end{proof}
\begin{figure}
\includegraphics[scale=0.27]{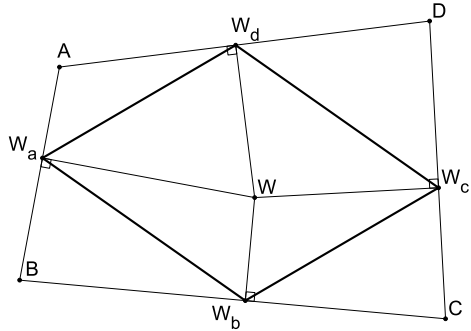}

\caption{\label{fig:Varignon parallelogram}The pedal parallelogram of W. }
\end{figure}

It is interesting to note the following
\begin{cor}
The pedal of $W$ with respect to $ACBD$ (i.e., a self-intersecting
quadrilateral whose sides are the two diagonals and two opposite sides
of the original quadrilateral) is also a parallelogram. 
\end{cor}
The Theorem above also implies that the pedal of $W$ is nondegenerate.
(We will see later that the pedal of $S$ degenerates to four points
lying on a straight line). While examples show that the pedal of $W$
and the Varignon parallelogram have different ratios of sides (and,
therefore, are not similar in general), it is easy to see that they
coincide in the case of a cyclic quadrilateral:
\begin{cor}
The Varignon parallelogram $M_{a}M_{b}M_{c}M_{d}$ is a pedal parallelogram
of a point if and only if the quadrilateral is cyclic and the point
is the circumcenter. In this case, $M_{a}M_{b}M_{c}M_{d}=W_{a}W_{b}W_{c}W_{d}$. \end{cor}
\begin{thm}
\label{thm:Pedal of W}The pedal quadrilateral of a point with respect
to quadrilateral $ABCD$ is a nondegenerate parallelogram if and only
if this point is $W$. \end{thm}
\begin{proof}
By Lemma \ref{lem:pedal of pt on CS}, if $P\in CS(o_{1},o_{3})\cap CS(o_{2},o_{4}),$
then both pairs of opposite sides of the pedal quadrilateral $P_{a}P_{b}P_{c}P_{d}$
are parallel. 

Assume that the pedal quadrilateral $P_{a}P_{b}P_{c}P_{d}$ of $P$
is a nondegenerate parallelogram. Since $P_{d}AP_{a}P$ is a cyclic
quadrilateral,
\begin{eqnarray*}
|P_{a}P_{d}| & = & \frac{|PA|}{2\sin\alpha},\\
|P_{b}P_{c}| & = & \frac{|PC|}{2\sin\gamma}.
\end{eqnarray*}
The assumption $|P_{a}P_{d}|=|P_{b}P_{c}$| implies that $|PA|:|PC|=\sin\gamma:\sin\alpha.$
Similarly, $|P_{a}P_{b}|=|P_{c}P_{d}|$ implies $|PB|:|PD|=\sin\delta:\sin\beta,$
so that $P$ must be on the Apollonian circle with respect to $A,C$
with ratio $\sin\gamma:\sin\alpha$ and on the Apollonian circle with
respect to $B,D$ with ratio $\sin\delta:\sin\beta$. These Apollonian
circles are easily shown to be $CS(o_{1}^{(0)},o_{3}^{(0)})$ and
$CS(o_{2}^{(0)},o_{4}^{(0)})$, the circles of similitude of the previous
generation quadrilateral. One of the intersections of these two circles
of similitude is $W$. Let $Y$ be the other point of intersection.
Computing the ratios of distances from $Y$ to the vertices, one can
show that the pedal of $Y$ is an isosceles trapezoid. That is, instead
of two pairs of equal opposite sides, it has one pair of equal opposite
sides and two equal diagonals. This, in particular, means that $Y$
does not lie on $CS(o_{1},o_{3})\cap CS(o_{2},o_{4})$. It follows
that $W$ is the only point for which the pedal is a nondegenerate
parallelogram.\end{proof}
\begin{rem*}
Note that another interesting pedal property of a quadrilateral was
proved by Lawlor in \cite{Lawlor four pedals 1,Lawlor four pedals 2}.
For each vertex, consider its pedal triangle with respect to the triangle
formed by the remaining vertices. The four resulting pedal triangles
are directly similar to each other. Moreover, the center of similarity
is the so-called \emph{nine-circle point}, denoted by $H$ in Scimemi's
paper \cite{Scimemi}.
\end{rem*}

\subsection{Simson line of a quadrilateral }

Recall that for any point on the circumcircle of a triangle, the feet
of the perpendiculars dropped from the point to the triangle's sides
lie on a line, called the \emph{Simson line }corresponding to the
point (see Figure \ref{fig:Simson line}). Remarkably, in the case
of a quadrilateral, Lemma \ref{lem:pedal of pt on CS} and Theorem
\ref{thm:Pedal of W} imply that there exists a unique point for which
the feet of the perpendiculars dropped to the sides are on a line
(see Theorem \ref{Quadrilateral Simson Line} below). 

In the case of a noncyclic quadrilateral, this point turns out to
be the second point of intersection of $CS(o_{1},o_{3})$ and $CS(o_{2},o_{4})$,
which we denote by $S$. For a cyclic quadrilateral $ABCD$ with circumcenter
$O$, even though all triad circles coincide, one can view the circles
$(BOD)$ and $(AOC)$ as the replacements of $CS(o_{1},o_{3})$ and
$CS(o_{2},o_{4})$ respectively. The second point of intersection
of these two circles, $S\in(BOD)\cap(AOC)$, $S\neq W$ also has the
property that the feet of the perpendiculars to the sides lie on a
line. Similarly to the noncyclic case (see Lemma \ref{lem:pedal of pt on CS}),
one can start by showing that the pedal quadrilateral of a point is
a trapezoid if and only if the point lies on one of the two circles,
$(BOD)$ or $(AOC)$. 

In analogy with the case of a triangle, we will call the line $S_{a}S_{b}S_{c}S_{d}$
the \emph{Simson line} and $S$ the \emph{Simson point} \emph{of a
quadrilateral}, see Fig.  \ref{fig:Simson pt line 4gon}. 

\begin{flushleft}
\begin{figure}
\includegraphics[scale=0.22]{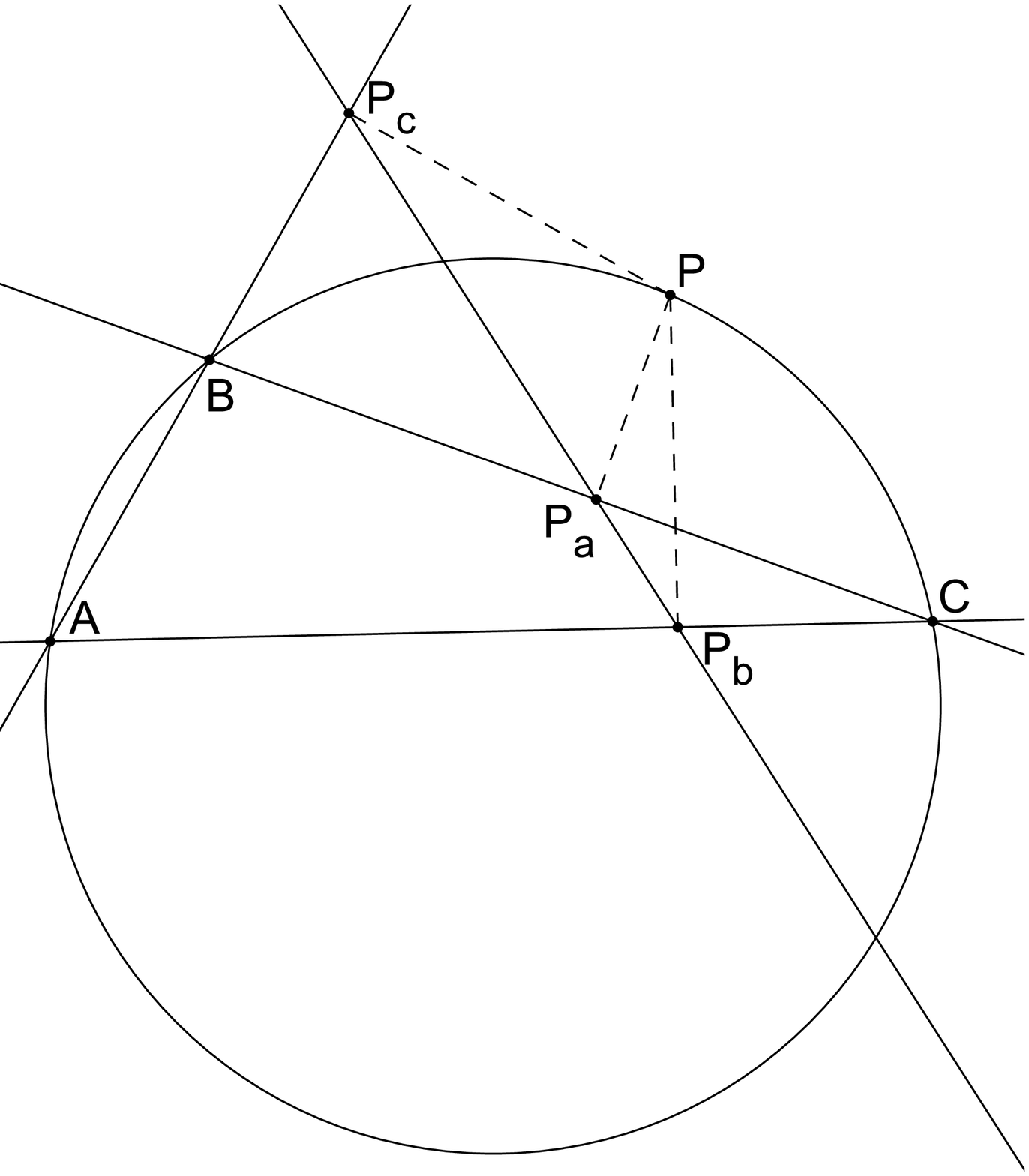}\includegraphics[scale=0.15]{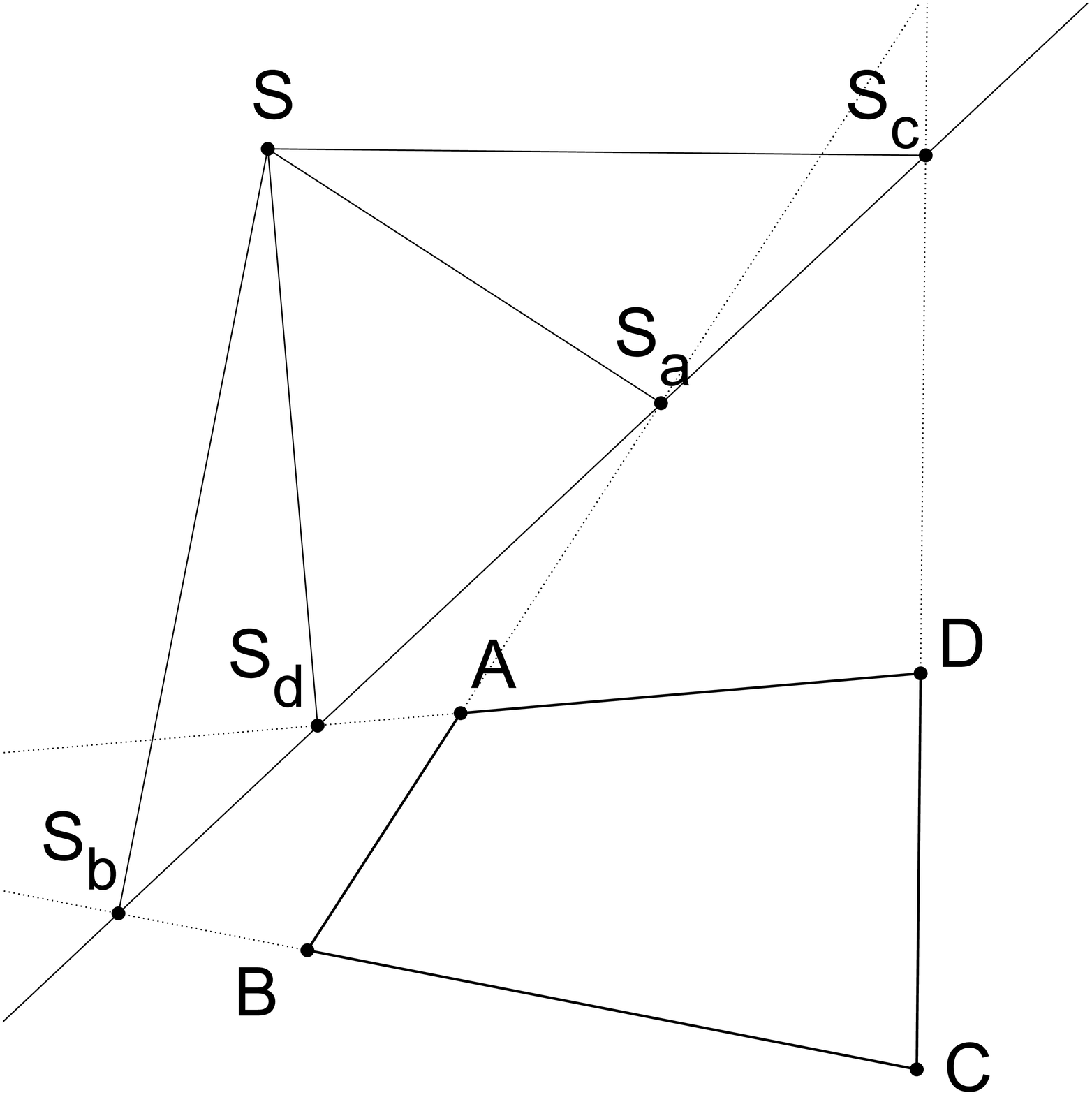}\caption{\label{fig:Simson line}A Simson line for a triangle and \label{fig:Simson pt line 4gon}the
Simson line of a quadrilateral.}

\end{figure}

\par\end{flushleft}
\begin{thm}
\label{Quadrilateral Simson Line}(The Simson line of a quadrilateral)
The feet of the perpendiculars dropped to the sides from a point on
the plane of a quadrilateral lie on a straight line if and only if
this point is the Simson point.
\end{thm}
Unlike in the case of a triangle, where every point on the circumcircle
produces a Simson line, the Simson line of a quadrilateral is unique.
When the original quadrilateral is a trapezoid, the Simson point is
the point of intersection of the two nonparallel sides. In particular,
when the original quadrilateral is a parallelogram, the Simson point
is point at infinity. The existence of this point is also mentioned
in \cite{Johnson-book}.

Recall that all circles of similitude intersect at $W$. The remaining
${6 \choose 2}=15$ intersections of pairs of circles of similitude
are the Simson points with respect to the ${6 \choose 4}=15$ quadrilaterals
obtained by choosing $4$ out of the lines forming the complete quadrangle.
Thus for each of the $15$ quadrilaterals associated to a complete
quadrangle  there is a Simson point lying on a pair of circles of
similitude.

\subsection{Isogonal conjugation with respect to a quadrilateral}

Recall that the isogonal conjugate of the first isodynamic point of
a triangle is the Fermat point (i.e., the point minimizing the sum
of the distances to vertices of the triangle). Continuing to explore
the analogy of $W$ with the isodynamic point, we will now define
isogonal conjugation with respect to a quadrilateral and study the
properties of $W$ and $S$ with respect to this operation. 

Let $P$ be a point on the plane of $ABCD$. Let $l_{A},l_{B},l_{C},l_{D}$
be the reflections of the lines $AP,BP,CP,DP$ in the  bisectors of
$\angle A$, $\angle B$, $\angle C$ and $\angle D$ respectively. 
\begin{defn*}
\label{def: isogonal conj for quad}Let $P_{A}=l_{A}\cap l_{B}$,
$P_{B}=l_{B}\cap l_{C},$ $P_{C}=l_{C}\cap l_{D},$ $P_{D}=l_{D}\cap l_{A}$.
The quadrilateral $P_{A}P_{B}P_{C}P_{D}$ will be called the \emph{isogonal
conjugate of $P$ }with respect to $ABCD$ and denoted by $Iso_{ABCD}(P)$.
\begin{figure}
\includegraphics[scale=0.24]{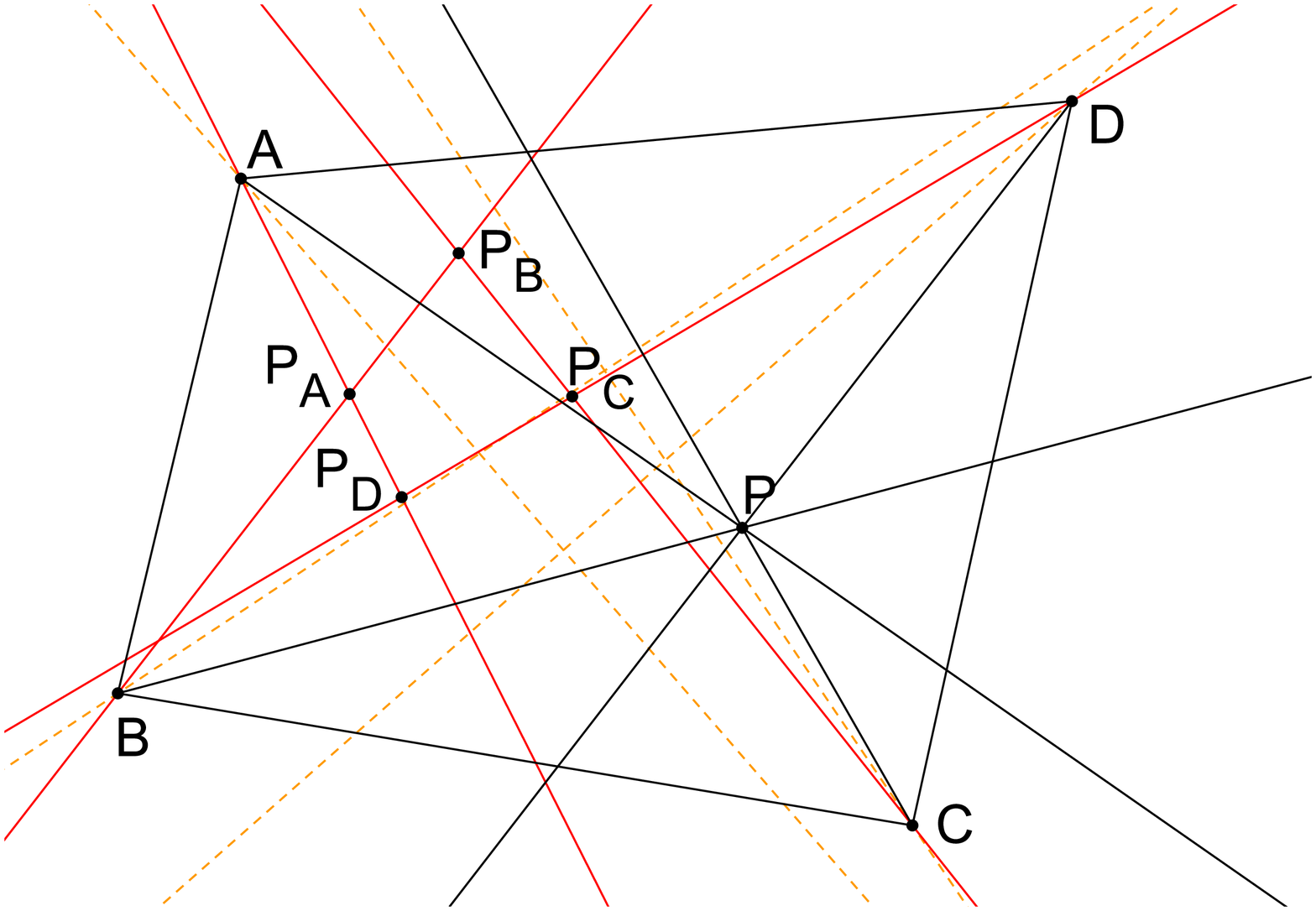}

\caption{Isogonal conjugation with respect to a quadrilateral}

\end{figure}

\end{defn*}
The following Lemma relates the isogonal conjugate and pedal quadrilaterals
of a given point:
\begin{lem}
\label{lem:isog and pedal} The sides of the isogonal conjugate quadrilateral
and the pedal quadrilateral of a given point are perpendicular to
each other.\end{lem}
\begin{proof}
Let $b_{A}$ be the  bisector of the  $\angle DAB$. Let $I=l_{A}\cap P_{a}P_{d}$
and $J=b_{A}\cap P_{a}P_{d}$. Since $AP_{a}PP_{d}$ is cyclic, it
follows that $\angle P_{d}AP=\angle P_{d}P_{a}P$. Since $PP_{a}\perp P_{a}A$,
it follows that $AI\perp P_{a}P_{d}$. Therefore, $P_{A}P_{D}\perp P_{a}P_{d}$.
The same proof works for the other sides, of course. 
\end{proof}
\begin{figure}
\includegraphics[scale=0.15]{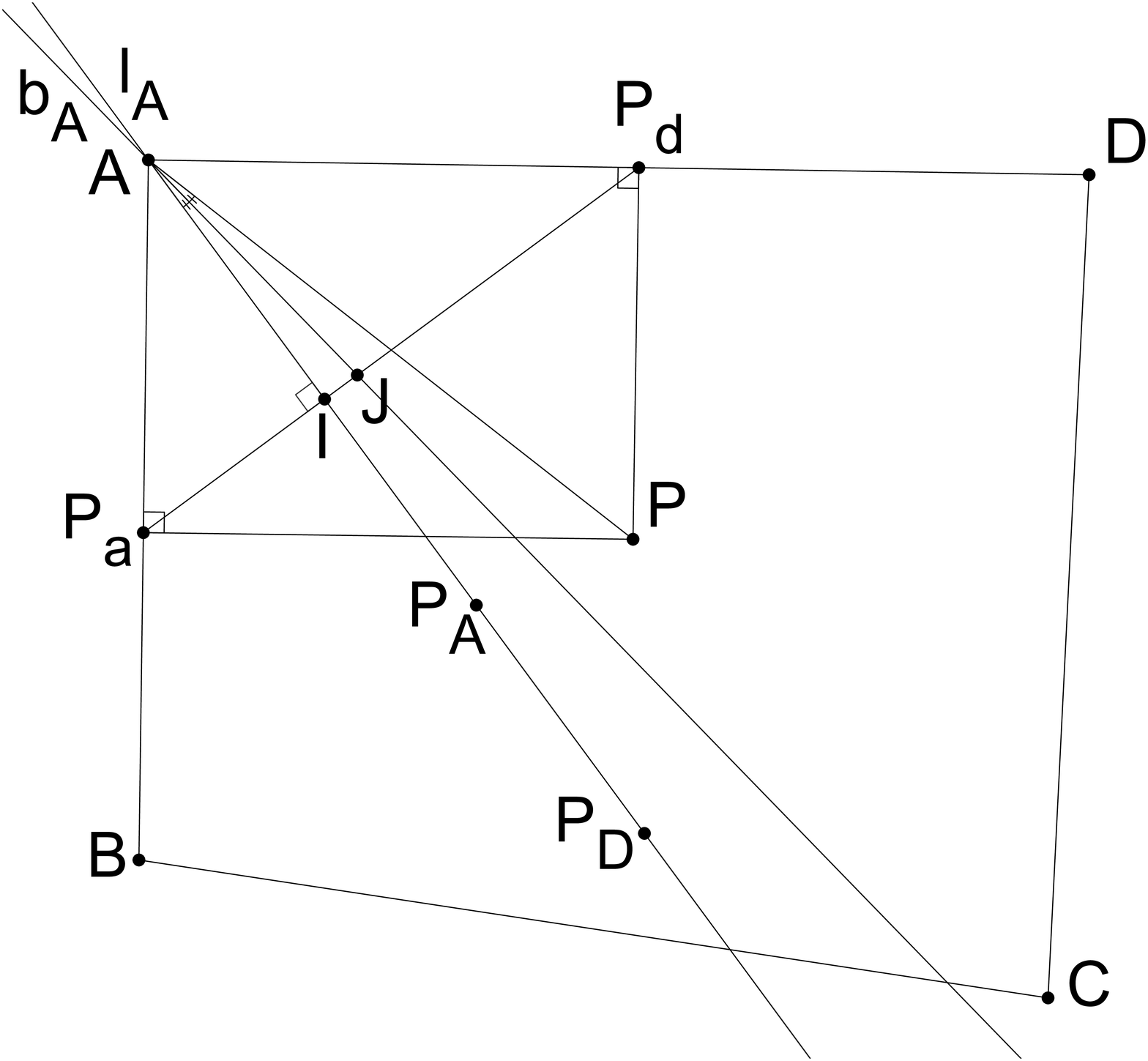}

\caption{Lemma \ref{lem:isog and pedal}.}
\end{figure}

The Lemma immediately implies the following properties of the isogonal
conjugates of $W$ and $S$: 
\begin{thm}
\label{thm: isogonal conjugates of W and S}The isogonal conjugate
of $W$ is a parallelogram. The isogonal conjugate of $S$ is the
degenerate quadrilateral whose four vertices coincide at infinity. 
\end{thm}
The latter statement can be viewed as an analog of the following property
of isogonal conjugation with respect to a triangle: the isogonal conjugate
of any point on the circumcircle is the point at infinity.

\subsection{Reconstruction of the quadrilateral}

The paper by Scimemi \cite{Scimemi} has an extensive discussion of
how one can reconstruct the quadrilateral from its central points.
Here we just want to point out the following 3 simple constructions:
\begin{enumerate}
\item Given $W$ and its pedal parallelogram $W_{a}W_{b}W_{c}W_{d}$ with
respect to $A_{1}B_{1}C_{1}D_{1}$, one can reconstruct $A_{1}B_{1}C_{1}D_{1}$
by drawing lines through $W_{a},W_{b},W_{c},W_{d}$ perpendicular
to $WW_{a},WW_{b},WW_{c},WW_{d}$ respectively. The construction is
actually simpler than reconstructing $A_{1}B_{1}C_{1}D_{1}$ from
midpoints of sides (i.e., vertices of the Varignon parallelogram)
and the point of intersection of diagonals. 
\item Similarly, one can reconstruct the quadrilateral from the Simson point
$S$ and the four pedal points of $S$ on the Simson line.
\item Given three vertices $A_{1},B_{1},C_{1}$ and $W$, one can reconstruct
$D_{1}$. Here is one way to do this. The given points determine the
circles $o_{2}=(A_{1}B_{1}C_{1})$, $CS(o_{2},o_{1})=(A_{1}WB_{1})$
and $CS(o_{2},o_{3})=(B_{1}WC_{1})$. Given $o_{2}$ and $CS(o_{2},o_{1}),$
we construct the center of $o_{1}$ as $A_{2}=\mbox{Inv}_{CS(o_{2},o_{1})}(B_{2})$
(see property \ref{enu:CS Inversion exchanges centers} in the Preliminaries
of Section \ref{sec: W}). Similarly, $C_{2}=\mbox{Inv}_{CS(o_{2},o_{3})}(B_{2}).$
Then $D_{1}$ is the second point of intersection of $o_{1}$ (the
circle centered at $A_{2}$ and going through $A_{1},B_{1}$) and
$o_{3}$ (the circle centered at $C_{2}$ and going through $B_{1},C_{1}$).
Alternatively, one can use the property that $D_{1}=\mbox{Iso}_{T_{2}}\circ\mbox{Inv}_{o_{2}}(W)$.\end{enumerate}

$ $

Olga Radko: Department of Mathematics, UCLA

\textit{E-mail address: radko@math.ucla.edu}

$ $

Emmanuel Tsukerman: Stanford University

\textit{E-mail address: emantsuk@stanford.edu}

\begin{thebibliography}{References}
\bibitem[1]{Baker} Baker, \emph{Principles of Geometry, }4 (1925),
p. 17.

\bibitem[2]{De Majo} A. De Majo, \emph{Sur un point remarquable du
quadrangle. }Mathesis, \textbf{63 }(1953), 236-240. 

\bibitem[3]{Mallison} H.V. Mallison, \emph{Pedal Circles and the
Quadrangle, }Math. Gazette, v. 42, No. 339 (Feb. 1958), pp. 17-20.

\bibitem[4]{Reflections 3}C.F. Parry, M.S. Longuet-Higgins, $(Reflections)^{3}$,
Math. Gazette, \textbf{59} (1975), 181-183.

\bibitem[5]{Wood} P.W. Wood, \emph{Points Isogonally Conjugate with
Respect to a Triangle, }The Math. Gazette, \textbf{25 }(1941), 266-272. 

\bibitem[6]{Scimemi} B. Scimemi, \emph{Central Points of the Complete
Quadrangle, }Milan. J. Math., \textbf{75} (2007), 333-356.

\bibitem[7]{Prasolov} V.V. Prasolov, \emph{Plane Geometry Problems},
vol. 1, problem 6.31, 1986. (in Russian). 

\bibitem[8]{Prasolov-translation} V. V. Prasolov, \emph{Problems
in Plane and Solid Geometry}, vol. 1 (translated by D. Leites), problem
6.31, available at http://students.imsa.edu/\textasciitilde{}tliu/Math/planegeo.pdf. 

\bibitem[9]{Langr}J.Langr, problem 1050E, \emph{Am. Math. Monthly,}
\textbf{60}, 1953. 

\bibitem[10]{D. Bennett}D. Bennett, \emph{Dynamic Geometry Renews
Interest in an Old Problem}, Geometry Turned On, MAA Notes 41, 1997,
25-28.

\bibitem[11]{King}J. King, \emph{Quadrilaterals Formed by Perpendicular
Bisectors}, in Geometry Turned On, MAA Notes 41, 1997, 29--32. 

\bibitem[12]{Shepard} G.C. Shephard, \emph{The perpendicular bisector
construction'', }Geom. Dedicata, \textbf{56}, 75-84 (1995). 

\bibitem[13]{cut-the-knot}A. Bogomolny, \emph{Quadrilaterals Formed
by Perpendicular Bisectors,} Interactive Mathematics Miscellany and
Puzzles, http://www.cut-the-knot.org/Curriculum/Geometry/PerpBisectQuadri.shtml,
Accessed 06 December 2010. 

\bibitem[14]{Neville} E.H. Neville, \emph{Isoptic Point of a Quadrangle,
}J. London Math. Soc. (1941), 3, pp. 173--174. 

\bibitem[15]{Grinberg}Darij Grinberg, \emph{Isogonal conjugation
with respect to triangle, }unpublished notes, http://www.cip.ifi.lmu.de/\textasciitilde{}grinberg/Isogonal.zip,
last accessed on 1/24/2010.

\bibitem[16]{Johnson-book} Roger A. Johnson, \emph{Advanced Euclidean
Geometry}, Dover Publications, Mineola, NY, 2007. 

\bibitem[17]{Coxeter} Coxeter, H. S. M. and Greitzer, S. L. \emph{Geometry
Revisited}, Washington, DC: Math. Assoc. Amer., 1967.

\bibitem[18]{Lawlor four pedals 1} J. K. Lawlor, \emph{Pedal Triangles
and Pedal Circles.} Math Gazette, \textbf{9 }(1917), pp. 127--130. 

\bibitem[19]{Lawlor four pedals 2}J. K. Lawlor, \emph{Some Properties
Relative to a Tetrastigm. }Math Gazette, \textbf{10} (1920), pp. 135-139. 

\bibitem[20]{Hyacinthos} Hyacinthos forum, message 6411, http://tech.dir.groups.yahoo.com/group/Hyacinthos/message/6411. 

\end{thebibliography}
\end{document}